\documentclass[11pt]{article}
\usepackage{fullpage}
\usepackage{amssymb,amsthm,amsmath,hyperref,graphicx,relsize,array,caption,subcaption}
\usepackage[usenames,dvipsnames]{color}

\newcommand{\e}{\mathrm{e}}
\newcommand{\eee}{\mathrm{e}}
\renewcommand{\k}{\kappa}
\newcommand{\bd}{\partial}
\newcommand{\eps}{\varepsilon}
\newcommand{\abs}[1]{\left| #1 \right|}
\DeclareMathOperator*{\Prob}{Prob}

\DeclareMathOperator*{\Exp}{\mathbb{E}}
\newcommand{\ExpCond}[2]{\Exp\left[#1 \mid #2\right]}
\DeclareMathOperator*{\depth}{depth}
\DeclareMathOperator*{\dist}{dist}
\newcommand{\rootvtx}{r}

\newcommand{\ie}{\emph{i.e.}}

\newcommand{\etal}{\emph{et al.}}
\newcommand{\Poi}{\textrm{Poisson}}

\newcommand{\ddy}{\frac{\mathrm{d}}{\mathrm{d}y}}
\newcommand{\dc}{\mathrm{d}c}
\newcommand{\dlambda}{\mathrm{d}\lambda}

\newtheorem{theorem}{Theorem}[section]
\newtheorem{lemma}[theorem]{Lemma}
\newtheorem{proposition}[theorem]{Proposition}
\theoremstyle{definition}
\newtheorem{definition}[theorem]{Definition}
\newtheorem{remark}[theorem]{Remark}

\begin{document}

\title{Spatial Mixing for Independent Sets in Poisson Random Trees}
\author{Varsha Dani\thanks{University of New Mexico, Dept. of Computer
  Science.  email: {\tt varshadani@gmail.com, hayes@cs.unm.edu}} \and Thomas P. Hayes$^*$ \and
  Cristopher Moore\thanks{Santa Fe Institute.  email: {\tt moore@santafe.edu}}}
\date{}
\maketitle

\thispagestyle{empty}

\begin{abstract}
We consider correlation decay in the hard-core model with fugacity $\lambda$
on a rooted tree $T$ in which the arity of each vertex is independently Poisson 
distributed with mean $d$. Specifically, we investigate the question of which 
parameter settings $(d, \lambda)$ result in strong spatial mixing, weak spatial mixing, 
or neither. (In our context, weak spatial mixing is equivalent to Gibbs uniqueness.)
For finite fugacity, a zero-one law implies that these spatial mixing properties hold either almost surely or almost never, once we have conditioned on whether $T$ is finite or infinite.

We provide a partial answer to this question, which implies in particular that
\begin{enumerate}
\item 
As $d \to \infty$, weak spatial mixing on the Poisson tree occurs whenever 
$\lambda < f(d) - o(1)$ but not when $\lambda$ is slightly above $f(d)$, where 
$f(d)$ is the threshold for WSM (and SSM) on the $d$-regular tree.
This suggests that, in most cases, Poisson trees have similar spatial mixing 
behavior to regular trees.
\item When $1 < d \le 1.179$, there is weak spatial mixing on the $\Poi(d)$ tree for all values of $\lambda$. However, strong spatial mixing does not hold for sufficiently
large $\lambda$.  This is in contrast to regular trees, for which strong spatial mixing and weak spatial mixing always coincide.
\end{enumerate}
For infinite fugacity SSM holds only when the tree is finite, and hence almost surely fails on the $\Poi(d)$ tree when $d>1$. We show that WSM almost surely holds on the $\Poi(d)$ tree for 
$d <  \e^{1/\sqrt{2}}/\sqrt{2} =1.434...$, but that it fails with positive probability if $d>\e$. 

\end{abstract}

\newpage
\setcounter{page}{1}
\section{Introduction}

Spatial mixing, or the decay of correlations between spins
in a spin system, is a fundamental question of interest in
statistical physics.  It is intimately related to temporal mixing
for the corresponding Glauber dynamics Markov chain, 
which means fast convergence to its equilibrium distribution.

There are two flavors of spatial mixing: strong and weak (see
Section~\ref{sec:WSM-SSM-prelims} for definitions.)  For our
purposes, weak spatial mixing is equivalent to Gibbs uniqueness,
another fundamental concept from statistical physics.

The hard core model defines a distribution
over the independent sets of a graph $G$ in terms of a fugacity
$\lambda > 0$.  When $G$ is finite and
$\lambda = 1$, this is the uniform distribution.
More generally, an independent set $S$ has a probability proportional
to $\lambda^{|S|}$, so that when $\lambda > 1$, the distribution
is biased towards larger independent sets, and when $\lambda < 1$,
it is biased towards smaller ones.  By convention, 
when $\lambda = +\infty$, the conditional 
distribution on finite subgraphs 
is uniform over all independent sets of maximum size.

In computer science, the problem of sampling from this distribution
when $\lambda = 1$ is well-known to be poly-time equivalent to
the problem of approximately counting the independent sets of a
graph, which is known to be a hard problem in general. 
We refer the reader to recent work by Sly and Sun~\cite{SlySun} 
for further hardness results.

A seminal paper of D. Weitz~\cite{weitz-tree} found that the
infinite regular $d$-ary tree has the same threshold for
weak and strong spatial mixing, namely 
$\lambda = d^d/(d-1)^{d+1} \sim e/d$.
More importantly, this is a worst case: every other
graph of maximum degree $d+1$ also exhibits WSM and SSM 
for all $\lambda$ up to the aforementioned threshold.
At the time, this established the strongest positive results for
spatial mixing for a wide variety of graphs, including, for instance,
the square grid.

Brightwell, H\"agrstr\"om and Winkler~\cite{brightwell} showed that there are graphs, 
even trees, for which the property of WSM is non-monotone as a function of $\lambda$.
That is, increasing $\lambda$ can actually \emph{decrease} the extent to which correlations 
travel over long distances, and so WSM holds at sufficiently small and sufficiently large 
$\lambda$, but not in between.  They even give a more complicated construction
(not a tree) for which the hard-core model exhibits WSM iff
$\lambda \in (0, \lambda_1] \cup [\lambda_2, \lambda_3]$
where $\lambda_1 < \lambda_2 < \lambda_3 < \infty$.
 
Restrepo \etal~\cite{RSTVY-grid} showed that for some graphs, such as the
planar square lattice, SSM occurs at higher $\lambda$ than for the $4$-regular tree.
Recently, Vera, Vigoda and Yang~\cite{VVY-grid} have shown that the tree of
self-avoiding walks on the square lattice contains a subtree which has
WSM but not SSM, at a still higher value of $\lambda$, but still below
the conjectured critical value for the square lattice.  (See
~\cite[Lemmas 4, 7]{VVY-grid}.)  This suggests that it may not
be such an uncommon phenomenon for WSM to occur without SSM.
In Section~\ref{sec:alternating} we exhibit an example of an infinite tree which 
has WSM for 
all $\lambda >0$ but does not have SSM for any $\lambda > 4$.

We consider random Poisson trees, in which every vertex has an
independent, identically Poisson distributed number of children.
This is a natural model because of its connection to sparse 
Erd\H{o}s-R\'enyi random graphs, $G(n,p)$.  When $d = \Theta(1)$ and $p = d/n$, 
for large $n$, the local structure of balls of volume $o(\sqrt{n})$ is well approximated 
by a Poisson tree. 

It is natural, given an infinite graph, to consider the following threshold
conjecture: There is a threshold $\lambda_{\mathrm{crit}}$ such that 
WSM holds if and only if $\lambda < \lambda_{\mathrm{crit}}$.
The analogous conjecture with SSM in place of WSM is also interesting.
For instance, both conjectures are known to be true with 
\[
\lambda_{\mathrm{crit}} = \frac{\Delta^{\Delta}}{(\Delta -
  1)^{\Delta+1}}
\]
when $G$ is the infinite regular $\Delta$-ary tree.  Note that
$\lambda_{\mathrm{crit}}$ is asymptotically $\eee/\Delta$ as 
$\Delta \to \infty$.
Brightwell, H\"aggstr\"om and Winkler~\cite{brightwell} have 
constructed other graphs $G$ for which the WSM conjecture
is false. 

Understanding weak spatial mixing for regular $d$-ary trees is relatively
straightforward.  Note that, in general, the conditional probability
$a_v$ that node $v$ is unoccupied, 
given that the parent of $v$ is unoccupied,
obeys the recurrence
\begin{center}
\begin{tabular}{m{0.5in} m{5in}}
\includegraphics[width=1.25in]{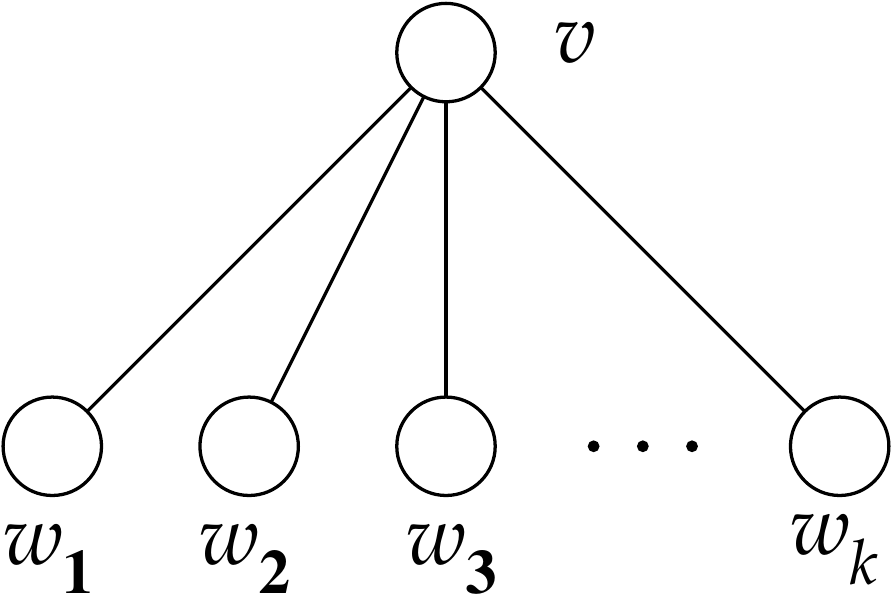}
& \larger[2]
\begin{equation}
\label{eq:bp}
a_v = \frac{1}{1+\lambda \prod_{w} a_w}
\end{equation}
\end{tabular}
\end{center}
where $w$ ranges over the children of $v$.  Since for the $d$-regular
tree all the $a_w$ are
equal, the problem boils down to understanding the stability of the
fixed point of the iterated function $f_d(a) = (1 + \lambda a^d)^{-1}$.

For random Poisson trees, the situation is more complicated.
Since the various subtrees of a node are no longer identical, 
but merely identically \emph{distributed}, we now need to,
in effect, consider a recurrence relation on \emph{distributions} 
rather than on real values.  

Intuitively, we may expect a $\Poi(d)$ tree to behave something like a
regular $d$-ary tree.  We show that this is the case for large $d$,
proving that WSM holds for $\lambda = c/d$ if $c < \eee$ but not if
$c > \eee$.  On the other hand, for small $d$, there are several
ways in which this is not the case.  In particular,
\begin{enumerate}
\item There are some settings of the Poisson parameter and fugacity
for which there is weak mixing (almost surely) but not strong mixing
(with positive probability).  In particular, this happens
when the expected degree is $1.1$ and
the fugacity is sufficiently large.
\item  For sufficiently small $d$, but still greater than $1$, the
Poisson tree exhibits WSM for all values of $\lambda$, even $\lambda = \infty$.
\item One might have thought that the phenomenon exploited in~\cite{brightwell},
where increasing $\lambda$ causes childless nodes to be occupied with
high probability, which then cuts off the flow of information from
their siblings up through their parent, is pathological.  In fact, we
will see that, for small enough $d$, this phenomenon is pervasive
in Poisson trees. 
\item As a consequence, some of our results are non-monotonic, in
that for $1.179 < d < 1.434$, we know WSM occurs at 
$\lambda = \infty$, and for sufficiently small $\lambda$, but
we don't know what happens in between.
\end{enumerate}

Before summarizing our main results, we begin by observing the following 
zero-one law for spatial mixing on Poisson trees with finite fugacity.  

\begin{theorem} \label{thm:zero-one}
For all $d > 1$ and $0 < \lambda < \infty$, conditioned on
$\Poi(d)$ being infinite, 
the probability that the hard-core model on $\Poi(d)$ with fugacity
$\lambda$ has WSM (resp. SSM) is either zero or one.
\end{theorem}
\noindent Note that for $d \le 1$, the Poisson tree $\Poi(d)$ is almost surely
finite.

In light of this zero-one law (proved in Section~\ref{sec:zero-one})
for finite $\lambda$,  we focus our attention on the question of which parameter settings 
$(d, \lambda)$ result in SSM, WSM, or neither.

We summarize our results for finite fugacities.  See
Figure~\ref{fig:graph} for graphs of some of the functions involved.
Overall, our results describe where WSM and SSM occur or do not occur 
in various regions of the $(d, \lambda)$ plane.

\begin{theorem} \label{thm:summary-finite}
The hard-core model with fugacity $\lambda <
\infty$ on $\Poi(d)$ has the following
properties, almost surely, conditioned on being infinite.
\begin{enumerate}
\item WSM if $d < 1.179...$, for any $0 < \lambda < \infty$. \label{part:1.179}
\item SSM if 
$\lambda < \begin{cases}\frac{4d^2}{(d^2-1)^2} & \mbox{when } d<\sqrt{2+\sqrt5}\\
\frac{3+\sqrt{1+4d^2}}{2d^2-4} & \mbox{otherwise.}
\end{cases}$ \label{part:SSM-1/d}
\item WSM if $\lambda < \frac{e - o(1)}{d}$, as $d \to \infty$, \label{part:WSM}
\item No WSM if $\lambda = \frac{e + o(1)}{d}$, as $d \to \infty$. \label{part:no-WSM}
\end{enumerate}
\end{theorem}

Thus, if the WSM threshold conjecture is true for the hard-core model on
the Poisson tree, then we have shown that the location of the
threshold is, for large $d$, asymptotically the same as for
$d$-regular trees.  On the other hand, unlike $d$-regular trees,
there is a range of parameters for which the Poisson tree exhibits 
WSM but not SSM.  Specifically, for $1 < d < 1.179...$ and for sufficiently
large $\lambda$, the Poisson tree almost surely has WSM but not SSM,
conditioned on being infinite; see Remark~\ref{rmk:noSSM}.
We conjecture that the Poisson
tree almost surely exhibits SSM up to a threshold that is asymptotically
$e/d$, the same as for $d$-regular trees. 

We also study spatial mixing properties of the $\Poi(d)$ tree when the fugacity is infinite. The following theorem summarizes our results for this case.

\begin{theorem} \label{thm:summary-infinite}
There exists a constant $d^* > 1$ such that for all $1 < d < d^*$, the hard-core model on 
$\Poi(d)$ with fugacity $\lambda = +\infty$ exhibits WSM but not SSM,
almost surely, conditioned on being infinite.
Futhermore, we prove that the largest such $d^*$ is at least
$e^{1/\sqrt{2}}/\sqrt{2} = 1.434...$, and at most $\eee = 2.718...$. 
\end{theorem}

\noindent We conjecture that $\eee$ is the correct value for $d^*$.
We prove Theorem~\ref{thm:summary-infinite} 
in Section~\ref{sec:inft-lambda}.

\begin{figure}[!ht]
  \centering
  \begin{subfigure}[b]{\textwidth}
    \centering
    \includegraphics[width=\textwidth]{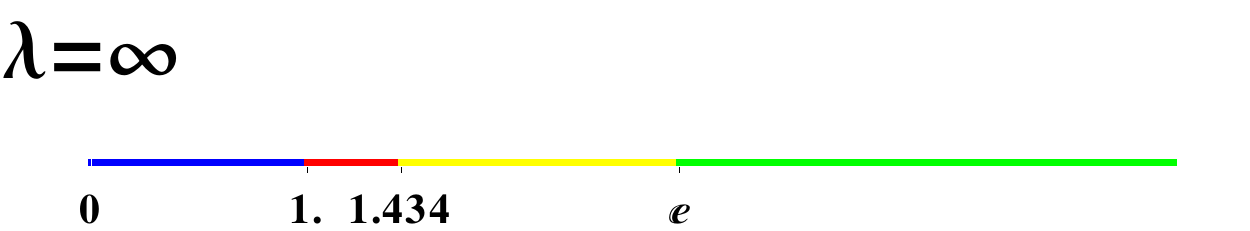}
    \caption{For $\lambda = +\infty$, there is SSM only when $d \le
      1$, in which case the tree is almost surely finite.
      There is WSM for $d < 1.434...$, but not for $d > \eee$.}
   \label{fig:plot-infty}
  \end{subfigure}
  \begin{subfigure}[b]{\textwidth}
    \centering
    \includegraphics[width=\textwidth]{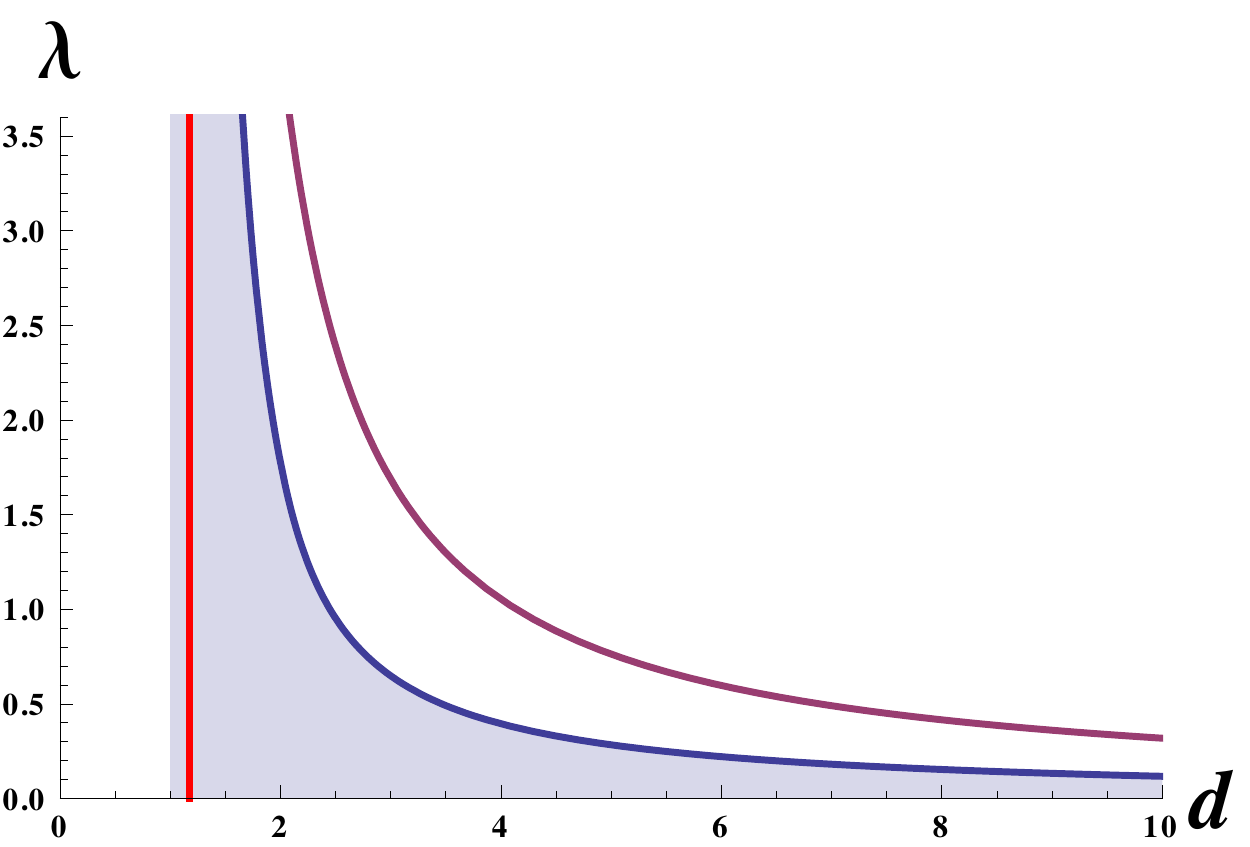}
    \caption{For finite $\lambda$, and $d > 1$ 
      there is SSM in the shaded region (to the left of the red curve
      $d = 1.179...$ and the blue curve that is asymptotic to $1/d$).
      The threshold for WSM is asymptotic to the purple curve, which
      is also the threshold for the regular $d$-ary tree.}
   \label{fig:plot-finite}
  \end{subfigure}
  \caption{Illustration for Theorems~\ref{thm:summary-infinite} and
    \ref{thm:summary-finite}.
  }
  \label{fig:graph}
\end{figure}

\section{Preliminaries}

\subsection{The Poisson Tree}

Let $d > 0$.  Consider a recursively generated random tree $T$, where we
sample a non-negative integer $X$ from the Poisson distribution with
mean $d$, namely, 
\[
(\forall i \ge 0) \quad \Prob(X = i) =\frac{ \eee^{-d}d^i}{i!}
\]
and define $X$ to be the number of children of the root of $T$.
Recursively, let each of these children be the root of a subtree
sampled independently in the same manner.  We call this the
\emph{Poisson tree of average arity $d$}, and denote it by
$\Poi(d)$.

For $d \le 1$, this tree is almost surely finite.  
For $d > 1$, the tree is infinite with positive probability,
but unlike an infinite $d$-regular tree, it has leaves: 
indeed, each non-root node has probability $\eee^{-d}$ to be a leaf.
(The root itself is a leaf with probability $\eee^{-d}(1+d)$, since it is a 
also a leaf if it has only one child.)

\begin{proposition} \label{prop:R2dR}
Let $d > 0$, and let $T$ be a Poisson tree $\Poi(d)$.  For $R \ge 0$, let
$f(R)$ denote the number of nodes in level $R$ of $T$.  Then, almost
surely,
\[
\lim_{R \to \infty} \frac{f(R)}{R^2 d^R} = 0.
\]
\end{proposition}

\begin{proof}
By Markov's inequality, for each $R$, we have
\(
\Prob \left( \frac{f(R)}{R^2 d^R} > R^{-1/2} \right) < R^{-3/2}.
\)
A union bound implies that there are almost surely only finitely many exceptions.
\end{proof}

\subsection{The Hard-Core Model (Independent Sets)}

In Statistical Mechanics, systems involving large numbers of
interacting particles are often modeled by a \emph{spin system}.
This is defined in terms of an underlying graph, often an infinite
lattice,
whose vertices are called \emph{sites}, each of which can be assigned a \emph{spin}
from some finite set $Q$.  A \emph{configuration} is a function
assigning a spin to each site.  A \emph{Gibbs measure} is a 
probability distribution over configurations that satisfies a 
consistency criterion on all finite ``patches'', or subsets of
vertices.   Specifically, for each finite subset $\Lambda \subset V$,
with boundary $\bd \Lambda = \{v \in \Lambda \mid \exists \{v,w\} \in
E_G, w \notin \Lambda \}$, and each 
\emph{boundary condition} $\sigma : \bd \Lambda \to Q$, 
the conditional distribution of the Gibbs measure restricted
to $\Lambda$, conditioned on agreeing with $\sigma$ on 
$\bd \Lambda$, is prescribed. 
Although it is known~\cite{Georgii} that a Gibbs measure
always exists, it is not, in general, guaranteed to be unique.
Indeed, many spin systems undergo a \emph{phase transition},
where some critical threshold for a defining parameter 
determines whether Gibbs uniqueness holds or not.

In the hard-core model, the spins correspond
to a site being ``occupied'' or ``unoccupied''.  Adjacent
sites are not allowed to both be occupied, and so configurations
are independent sets of the graph.  Configurations have probabilities
that are exponential in the number of occupied sites: an independent
set $S$ has probability $\frac1Z \lambda^{|S|}$, where $\lambda > 0$
is a parameter of the system called the \emph{fugacity}, and the
normalizing constant $Z$ is called the \emph{partition function}.

We will also be concerned with the case $\lambda = +\infty$, in which, on finite
patches, the prescribed distribution is considered to be 
uniform over all independent sets of
the maximum possible size.

\subsection{Weak and Strong Spatial Mixing} \label{sec:WSM-SSM-prelims}

``Spatial mixing'' refers to a phenomenon wherein
correlations between spins decay 
as the distance between the vertices increases.

Let $\Lambda$ be any set of vertices, let $\Psi \supset \Lambda$ be
a containing set of vertices
and let $\sigma, \tau : \bd \Psi \to Q$ be two boundary
configurations for the larger set.  
We are interested in the total variation distance
between the marginal distributions on configurations 
over $\Lambda$, conditioned on agreeing with $\sigma$ or $\tau$.
Now, consider infinite families of such triples $(\Psi, \sigma,
\tau)$, indexed by the positive integers.
If 
\[
\dist(\Lambda, \bd \Psi) \to \infty  \mbox{ implies }
\| \mu_\Psi^\sigma - \mu_\Psi^\tau \|_\Lambda  \to 0,
\]
then we say that \emph{weak spatial mixing} (WSM) holds.
If 
\[
\dist(\Lambda, \sigma \oplus \tau) \to \infty \mbox{ implies }
\| \mu_\Psi^\sigma - \mu_\Psi^\tau \|_\Lambda  \to 0,
\]
where $\sigma \oplus \tau$  denotes the set of vertices on which
$\sigma$ and $\tau$ disagree, 
then we say that \emph{strong spatial mixing} (SSM) holds.

Intuitively, weak spatial mixing requires the effect of changing
some spins to decay with distance, assuming all closer vertices
are unconstrained, while strong spatial mixing requires the effect
to decay even when some of the closer vertices are ``frozen'' in
an adversarial way (which must be the same for both boundary
conditions).  Obviously, SSM implies WSM.

The above definition of weak spatial mixing is easily seen to be 
equivalent to Gibbs uniqueness 
(see~\cite[Proposition 2.2]{weitz-gibbs}). 
We note that several alternative definitions of spatial mixing appear
in the literature.  In some of these, the rate of decay of correlation is
required to be exponential in the distance, rather than merely tending
to zero.  All of our results apply in this setting as well.  
In some definitions of spatial mixing, one either restricts attention
to the effect on a single vertex, \ie, $\Lambda = \{v\}$, 
and/or one restricts attention to boundary conditions that disagree on
a single boundary vertex.  In the case when the convergence
rate is required to be exponential, and moreover the graph is such that 
boundary sizes grow subexponentially, the restriction to a single 
disagreement doesn't matter (by a union bound).  For trees, however,
boundary sizes often grow exponentially, in which case the
specific rate of exponential decay of the effect of a vertex would matter.
On the other end, there are spin
systems where restricting $\Lambda$ to be a singleton makes
WSM hold trivally, even when it does not hold for larger sets
$\Lambda$.
\footnote{Here is a rather contrived example.
Start with any 2-spin system for which WSM does not hold.  
Replace each vertex with a pair of vertices, and decree 
that if the original vertex had spin 1, the
pair have the same spin, but uniformly random 1 or 2.  
If the original vertex had spin 2, the pair have opposite spins, again 
uniformly random.  We omit the details.}

\begin{remark} \label{rmk:SSM=WSM-forall-subgraphs}
In the case of  independent sets, it is well known that 
SSM on a graph $G$ is equivalent to WSM on all subgraphs of $G$, 
because any boundary vertices that are frozen to be unoccupied
can equivalently be deleted, and any that are occupied can 
equivalently have all their neighbors deleted.
\end{remark}

For independent sets on a tree, there is a simpler characterization
of spatial mixing in terms of \emph{non-occupation} probabilities.
Specifically, let $T$ be a finite tree with a designated root
vertex $\rootvtx$.  For each vertex $v$, let $a_v$
denote the conditional probability that $v$ is unoccupied, 
conditioned on $v$'s parent (if any) being unoccupied.
These non-occupation probabilities satisfy the recurrence \eqref{eq:bp}.

When $T$ is an infinite rooted tree, we will suppose that
an adversary has set arbitrary values $a_z \in [0,1]$
at level $R+1$.  In this case, we 
treat \eqref{eq:bp} as a recursive definition for $a_v$,
where $v$ is at distance $\le R$ from the root.
If for all sequences of boundary conditions, as $R \to \infty$,
$a_v$ converges to a well-defined limit $a^*_v$,
then we call $a^*_v$ the non-occupation probability of $v$.

Since the righthand side of \eqref{eq:bp} is a decreasing function
of each of the $a_w$, it follows by induction that, for any radius $R$, 
the extreme values of any $a_v$ are induced by the all-zeros
and the all-ones boundaries.  Thus, when proving the existence
of $a^*_v$, it suffices to consider boundary conditions of this type.

\begin{proposition} \label{prop:TFAE-WSM}
For the hard-core model on any infinite tree, the following are equivalent:
\begin{enumerate}
\item For all vertices $v$, there is a well-defined non-occupation
probability $a^*_v$.
\item Weak spatial mixing occurs.
\item There is a unique Gibbs distribution.
\end{enumerate}
Furthermore, when the fugacity, $\lambda$, is finite, 
this condition is equivalent to the three above:
\begin{enumerate}
\item[4.] For the root $\rootvtx$, there is a well-defined
  non-occupation probability $a^*_\rootvtx$.
\end{enumerate}
\end{proposition}

\begin{proof}[Proof sketch]
The equivalence of statements 2 and 3 is shown in~\cite[Proposition 2.2]{weitz-gibbs}. 

To see that statement 3 implies statement 1, let $v$ be any vertex in
the tree.  By Gibbs uniqueness, if we consider larger and larger balls
centered at $v$, the effect of the boundary configuration goes to
zero, and there is a well-defined marginal distribution on the spins
of $v$ and its parent.  Essentially by definition, $a^*_v$ must equal
the probability that $v$ is unoccupied, conditioned on its
parent being unoccupied.  Note that the effect of all spins outside
the subtree under $v$ can only influence the spin of $v$ through
the spin of its parent, which we have conditioned on.

To see that statement 1 implies statement 3, suppose for contradiction
that there were two distinct Gibbs measures.  Then their marginals
must differ on some finite patch $\Lambda \subset V$.  Starting with
the root vertex $v_0$, let $v_0, v_1, v_2, \dots, $ be a breadth-first
traversal of the tree.  Then, for some configuration $\sigma$, and
some finite $i$, the probability of $\sigma$ restricted to $v_0,
\dots, v_i$, must differ under the two Gibbs measures.  Choose
$i$ to be minimal with respect to this property.  In this case,
$\Prob( \sigma(v_i) \; \mid \; \sigma(v_0), \dots, \sigma(v_{i-1}) )$
differs under the two distributions.  The only way this can happen
is if the parent of $v_i$ is unoccupied under $\sigma$, in which
case the above conditional probability must equal $a^*_{v_i}$ in both
measures, a contradiction.

Statement 4 is a special case of statement 1, corresponding to weak spatial mixing at the root (since the root has no parent).  Hence statement 1 implies statement 4.
Statement 4 implies statement 1 when $\lambda$ is finite, because the recurrence 
\eqref{eq:bp} holds at every vertex $v$, under every boundary condition.  It follows that
the limit $a^*_v$ cannot exist unless the limits $a^*_w$ exist for every child vertex $w$.
\end{proof}

As before, note that for infinite $\lambda$, recurrence~\eqref{eq:bp} doesn't hold. Indeed,
it is possible for $a^*_v$ to be completely determined by a finite collection of its descendants.
For instance, if two children of $v$ are themselves childless, then $a^*_v = 1$ regardless
of any other consideration.  Thus statement 4 is weaker than
statements 1 through 3 when $\lambda = \infty$.

\subsection{Zero-One Law} \label{sec:zero-one}

In this section, we prove
Theorem~\ref{thm:zero-one}.  To this end, we say that
a boolean predicate, $S$, defined on rooted trees has property $\mathcal{R}$
if, for every tree $T$, $S(T)$ holds if and only if
$S(T')$ holds for every induced proper subtree $T'$ of $T$.
Note that any predicate with property $\mathcal{R}$ must hold
for every finite tree, by induction. 

\smallskip\par\noindent{\bf Examples:}  
\begin{enumerate}
\item ``$T$ is finite'' has property $\mathcal{R}$.
\item When $\lambda < \infty$, the property ``The hard-core model
for $T$ has WSM,'' has property $\mathcal{R}$, in light of
Proposition~\ref{prop:TFAE-WSM}.
\item Similarly, for $\lambda < \infty$, ``The hard-core model for $T$ has
SSM''  also has property $\mathcal{R}$.
\end{enumerate}

\begin{lemma} \label{lem:property-R-zero-one}
Let $A$ be a predicate with property $\mathcal{R}$.
Then, for a random $\Poi(d)$ tree, conditioned on being infinite, 
the conditional probability that $A$ holds is either zero or one.
\end{lemma}

\begin{proof}
Let $p$ denote the probability that $A(T)$ holds.  
Since $A(T)$ holds iff $A(T')$ holds for each of the
top-level subtrees $T'$ of $T$, and the number of such
subtrees is Poisson distributed with mean $d$, we have
\[
p = \sum_{i \ge 0} e^{-d} \frac{d^i}{i!} p^i = e^{d(p-1)}.
\]
This equation is easily seen to have the following solutions.
$p = 1$ is always a solution.
When $d \le 1$, this is the only solution in $[0,1]$.
When $d > 1$, there is a second solution $p^* < 1$, 
which equals the probability that $\Poi(d)$ is finite.
Since predicates with property $\mathcal{R}$ hold for
all finite trees, it follows that $p = p^* + (1-p^*) q$, where $q$ is the
conditional probability of $A(T)$ conditioned on $T$
being infinite.  Hence $q$ is $0$ or $1$, completing
the proof.
\end{proof}

Theorem~\ref{thm:zero-one} follows as an immediate corollary
in light of the above observation that having WSM (resp. SSM)
is a predicate with property $\mathcal{R}$.

\subsection{Alternating Trees}\label{sec:alternating}

Consider the infinite rooted tree $T$ with alternating layers of degree $d >12$ 
and degree 2 vertices, \ie, the root has $d$ children, each of whom have 
two children, each of whom have $d$ children and so on.
In this section, we examine the question of weak spatial mixing for
such trees.  Notice that,
since $T$ contains a complete binary tree, it does not have SSM for
$\lambda > 4$.

\begin{theorem}
$T$ has weak spatial mixing for all $\lambda \le \frac{d}{4 \ln d}$.
\end{theorem}
\begin{proof}
Consider the function 
\[
g(x) = \frac{1}{1+\lambda\left(1+\lambda x^2\right)^{-d}}
\]
which determines the values $a_w$ of the nodes $w$ at an even depth 
$r$ when the values of $a_w$ for $w$ at depth $r+2$ have been set to $x$.
Since $f$ is the composition of two monotone decreasing functions, it is 
monotone increasing. 
Observe that 
\[
g'(x) = -g(x)^2 \cdot \lambda (-d) \left(1+\lambda x^2\right)^{-d-1} \cdot 
2 \lambda x  = 2 x g(x)^2 \lambda^2 d\left(1+\lambda x^2\right)^{-d-1}.
\]
At the boundary, the adversary can set the $a_w$s to any values in $[0,1]$. 
However, recall that by \eqref{eq:bp}, for every level above that, these 
values will lie in $[\frac{1}{1+\lambda}, 1]$. Thus $x$ and $g(x)$ 
are both
between  $\frac{1}{1+\lambda}$ and 1, and we have
\begin{align*}
|g'(x)| &\le 2 \lambda^2 d \left(1+ \frac{\lambda}{(1+\lambda)^2}\right)^{-d-1}  \\
& \le 2 \lambda^2 d \e^{-\lambda(d+1)/(1+\lambda)^2}
\end{align*}
Since $d$ is large, when $\lambda < \frac{d}{4 \ln d}$, $|g'(x)|$ is bounded 
below 1 for all $x$ in $[\frac{1}{1+\lambda}, 1]$. It follows that 
$g$ is a 
contraction mapping with a unique fixed point $a^*$ in 
$[\frac{1}{1+\lambda}, 1]$, and moreover, for any 
$x \in [\frac{1}{1+\lambda}, 1]$, the sequence $\{a_n\}$ defined recursively by 
\[
a_0 = x; \,\, a_n = g(a_{n-1})
\]
converges to $a^*$.

Now suppose that the adversary sets the values of all the nodes at depth $R$ 
to be either all 0s or all 1s. Then applying \eqref{eq:bp} results in the same 
values at all the nodes at depth $R'$ which is the deepest even level above 
$R$. applying the function $g$ repeatedly from then on, we see that as $R$ 
goes to infinity, the value at the root, $a_\rootvtx$ converges to 
$a^*$. By the monotonicity of \eqref{eq:bp} with respect to each $a_w$, 
$a_\rootvtx$ converges to $a_*$ for all settings of the nodes at depth 
$R$ by the adversary. It follows that $T$ has weak spatial mixing for all 
$\lambda < \frac{d}{4 \ln d}$.
\end{proof}

On the other hand, $T$ contains the $2$-regular tree as a subtree. Thus $T$ 
does not have strong spatial mixing for any $\lambda > 4$.  Thus there is a 
large range of $\lambda$ for which it has weak, but not strong, spatial mixing.

Now consider the infinite tree $T'$ all of whose vertices at depth $r$ have 
$d(r)$ children, where
\[
d(r)= \begin{cases} 2 & \mbox{ if $r$ is odd}\\
                    2^{r+1} & \mbox{ if $r$ is even} \end{cases}
\]
(or any increasing function of $r$ on the even levels should be fine.) 
As before $T'$ contains the complete infinite binary tree as a subtree, and 
so has no strong spatial mixing above $\lambda = 4$. However, it is easily seen that 
$T'$ has weak spatial mixing for all $\lambda$.

\section{Infinite Fugacity: Maximum Independent Sets}\label{sec:inft-lambda}

In this section we derive upper and lower bounds on the Weak Spatial Mixing threshold in the infinite fugacity case. We note that at infinite fugacity,  the Poisson tree with average degree $d$ does not exhibit 
\emph{Strong} Spatial Mixing unless $d<1$ in which case the tree is almost certainly finite.

When $\lambda = \infty$, equation \eqref{eq:bp} is 
potentially indeterminate, so a good first step would be to re-examine the definition of
the model.  The defining notion is that, for any finite patch with
boundary condition, the distribution should be uniform over
independent sets of the maximum possible size.  However, 
in order to understand whether this condition leads to a unique Gibbs measure, 
we still want a recurrence for the probabilities $a_v$, that $v$ is
unoccupied, conditioned on its parent being unoccupied.

There are a couple of
good ways to deal with the indeterminism in \eqref{eq:bp}.
First, we can do arithmetic in the ring $\mathbb{R}[\lambda^{-1}]/(\lambda^{-2})$,
where we treat $\lambda^{-1}$ as an infinitesimal, that can be
ignored when added to any non-zero real number, and whose square is
treated as zero.  The expression
$1/(1+\lambda \prod_w a_w )$ evaluates to:
\begin{enumerate} 
\item $1$ whenever two or more of the $a_w$ are infinitesimal,
\item $\lambda^{-1} \prod_w  a_w^{-1}$ if none of the $a_w$ are
  infinitesimal, and
\item $1/(1 + c_{w'} \prod_{w: w \ne w'} a_w)$ if exactly one vertex $w'$
  has the infinitesimal value $a_{w'} = c_{w'} \lambda^{-1}$.
\end{enumerate}
The second approach is to treat the above infinitesimals as zeros,
but to reconstruct the coefficient $c_{w'}$ in case 3, from the
values on the children of $w'$.  This gives the formula
\[
a_v = \frac{\prod_z a_z}{\prod_z a_z + \prod_{w \ne w'} a_w},  \label{eq:bpi}
\]
where $z$ ranges over the children of the unique child $w'$ with
$a_{w'} = 0$.

We will refer to vertex $v$ as ``large'' when $a_v$ evaluates to a
non-zero real number, and as ``small'' when it evaluates to an
infinitesimal (or zero, if you prefer that viewpoint).  There is a 
third possibility, namely that no finite piece of the tree suffices
to determine whether $a_v$ is large or small, because of infinite
descent; in this case, we say $a_v$ is ``unlabeled.''  Our rules
above now give a particularly easy recursive description of when
a node is large, small, or unlabeled:
\begin{itemize}
\item[a.] If one or more children of $v$ is small, then $v$ is large. 
\item[b.] If all children of $v$ are large, then $v$ is small. 
\item[c.] Otherwise, no child of $v$ is small, and at least one child is
  unlabeled.  In this case, $v$ is unlabeled.
\end{itemize}
\begin{figure}
   \begin{center}
      \includegraphics[width=6.0in]{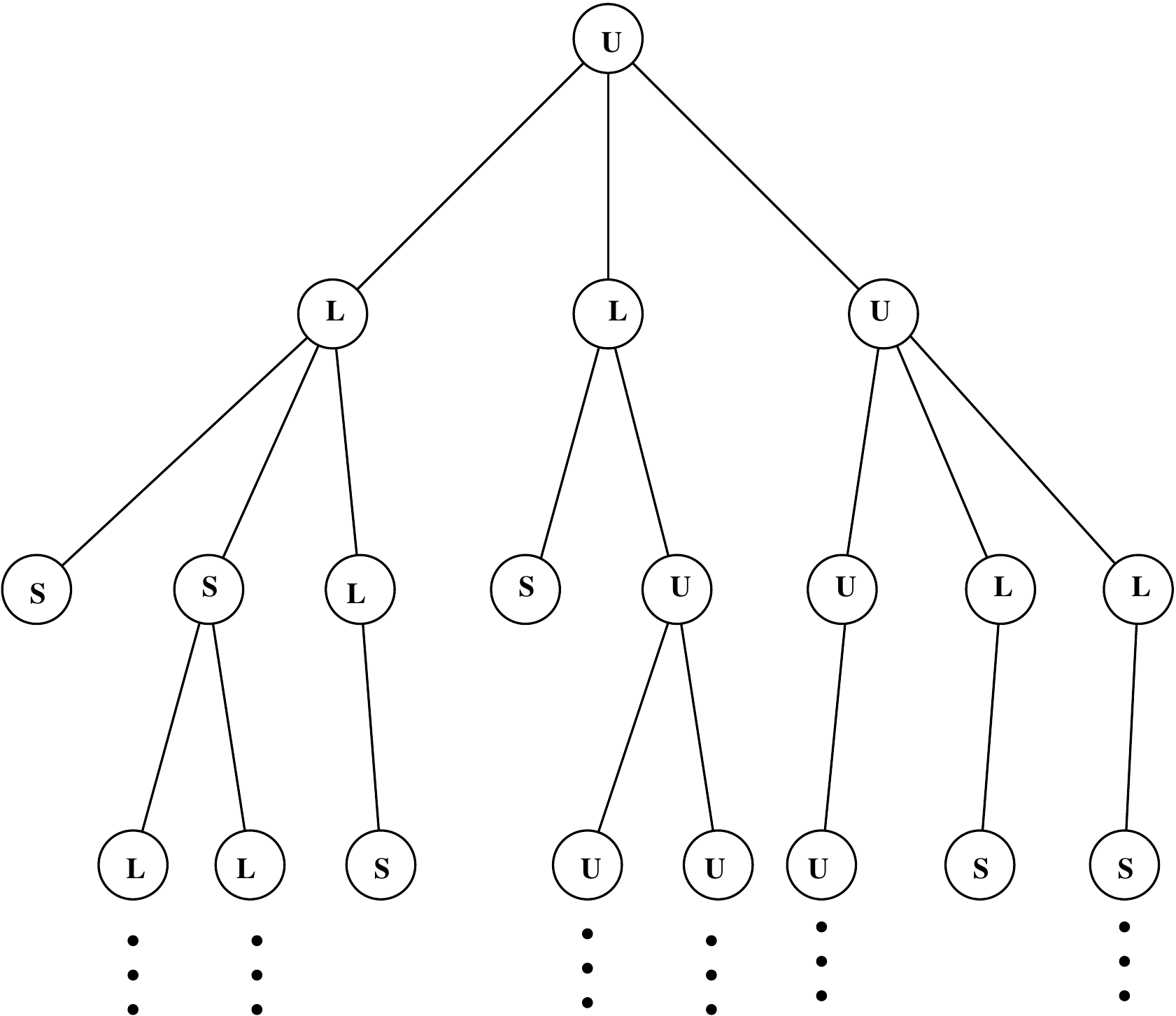}
      \caption{An example of Karp-Sipser labeling.}
      \label{fig:small-large-unlabeled}
   \end{center}
\end{figure}
We call this process \emph{Karp-Sipser labeling}, since it is a 
bottom-up version of the Karp-Sipser algorithm~\cite{ks}, which generates an independent 
set $S$ in a graph by choosing a vertex $v$ of degree 1 or 0, placing $v$ in $S$, 
and removing $v$ and its neighbor, if any,  from the graph. 
See Figure~\ref{fig:small-large-unlabeled}.

Starting from the leaves, which are small, one can work upward 
through the tree, using rules 1 and 2 to assign labels to all the
small and large nodes.  The nodes that remain unlabeled after
this (infinite) process are the ones we called ``unlabeled'' above.
It is easy to see that, by induction, each unlabeled node sits on
top of an infinite leafless subtree of unlabeled nodes.  The
unlabeled nodes in this tree may also have additional children
that are labeled large, who in turn have other children, about
which we are not concerned.

Now, suppose we cut off our tree at depth $R$, and set a pair of
boundary conditions on these nodes, that respects the labeled 
nodes,
and either occupies all or none of the remaining boundary 
nodes.  More precisely, under the first boundary condition, the
occupied nodes at depth $R$ are exactly the ones labeled "small,"
while under the second boundary condition, the unlabeled nodes
are also occupied.

In this case, it is easy to see by induction that, subject to this new
boundary condition, all the labeled nodes at depth $< R$ will keep
their original labels, and therefore the previously unlabeled
nodes at depth $i < R$ will either be all large or all small,
depending on the parity of $R - i$ and which of the two boundary
conditions was set.  

Now let $p_S$, $p_L$ and $p_U$ denote the probabilities that the root 
is labeled `small', `large' or `unlabeled' respectively. We have 
\[
p_S + p_L +p_U =1
\]

Then, by rules a,  b, and c  above, $p_S$ is the probability that  all the root's children are 
large, while $p_L$ is the probability that at least one child of the root is small.
Since each child is the root of an independently random subtree, which is distributed just 
as the entire tree is, the number of children that are large or small or unlabeled is Poisson-distributed 
with mean $dp_L$ or $dp_S$ or $dp_U$ respectively. This gives
\begin{equation}
\label{eq:plarge-psmall}
p_L = 1-\e^{-d p_S}   \ \ \ \mbox{ and } \ \ \  p_S = \e^{-d(p_S + p_U)} = \e^{-d(1-p_L)}
\end{equation}
Together, these imply
\[
p_S = \e^{-d\e^{-dp_S}}
\]
Letting $f$ denote the function $f(x) = \e^{-dx}$, we see that $p_S$ is a fixed point of $f \circ f$.  One fixed point of $f \circ f$ is the (unique) fixed point of $f$. 
Using Lambert's $W$ function, where $z = W(z)\e^{W(z)}$, this fixed point can be written 
as $W(d)/d$. In fact, when $d \le e$ this is the only real fixed point. In that case,
\[
p_S = W(d)/d,  \ \ \mbox{ and } \ \  p_L = 1 - f(p_S) = 1 - p_S
\]
so that $p_U =0$ and the root is labeled with probability 1. 
When $d > e$, on the other hand, the smallest real fixed point of $f\circ f$ is strictly smaller than $W(d)/d$, and is not a fixed point of $f$. In that case, the smallest fixed point is $p_S$, and hence $p_U >0$,  \ie,  with constant probability the root remains unlabeled.
 
We remark that all this corresponds exactly to the rigorous results on the Karp-Sipser algorithm~\cite{ks, ks-revisited}. On $G(n, p = d/n)$, if $d < e$ then the algorithm finds a maximal independent set, except for a core that consists w.h.p. of $O(\log n)$ vertex-disjoint cycles.

We are now ready to prove our upper and lower bounds.

\subsection{Upper Bound}

In this section we will analyze the situation when $\lambda = \infty$ and $d>\e$. 
Recall that in this case the root has positive probability $p_U$ to be unlabeled. Moreover, 
regardless of the root's label,  the number of children of the root that are, respectively, 
small, large and unlabeled are independent Poisson random variables with parameter, 
respectively, $d p_S$, $d p_L$ and $d p_U$.

It follows that  with positive probability,  the root is unlabeled and has 
at least two unlabeled children (and no small children.)
In this event, based on the parity of $R$, one boundary condition at depth $R$  
forces both those unlabeled  children to be occupied while the other forces them 
both to be unoccupied. Since the independent set must 
be of maximum size, if both are occupied then the root is forced to be unoccupied, 
while  if both are unoccupied, the root is forced to be occupied. Since these two 
alternatives remain possible, independent of $R$, 
there is no weak spatial mixing at the root.
We have shown the following, which implies the second half of 
Theorem~\ref{thm:summary-infinite}
\begin{theorem}
For $\lambda = \infty$ and $d > \e$, with positive probability, the $\Poi(d)$ tree does not have WSM at the root.
\end{theorem}

We remark that if the Poisson tree is infinite, it almost surely contains \emph{some} node which is unlabeled and has at least two unlabeled children.

\subsection{Lower Bound}

In this section we analyze the situation when $d\le \e$. Recall that in this case 
$p_S = W(d)/d$ was the unique fixed point of $f(x)= e^{-dx}$, so that 
\begin{equation}\label{eq:ps}
p_S = e^{-dp_S}
\end{equation}
and by~\eqref{eq:plarge-psmall}, $p_L = 1-e^{-dp_S}  = 1-p_S$, \ie, the root is labeled as either `small' or `large' with probability 1. Moreover, this labeling obeys the rules that 
\begin{itemize}
\item all children of a small node are large, and
\item at least one child of a large node is small.
\end{itemize}

To what do these labels correspond? Intuitively, a vertex being labeled `small' or `large'
respectively, corresponds to having non-occupation probabilities (as the root of its subtree) 
that are small or large respectively. For finite $\lambda \gg 1$, roughly speaking, this means 
$O(1/\lambda)$ or $\Theta(1)$ respectively.  Note however, that this intuition can sometimes be incorrect, for instance a node with very many children, all ``large,'' may have a large non-occupation probability, even though it receives a label of ``small."  Another example where the above intuition fails is for nodes at the root of a subtree isomorphic to a very long path,
specifically one of length $\omega(\sqrt{\lambda})$.  Although the nodes in this path are labelled with alternating ``small" and ``large" labels, actually almost all the conditional non-occupation probabilities will be approximately $1/\sqrt{\lambda}$.

When $\lambda$ is infinite, this becomes 
a distinction of zero vs. non-zero. In other words, conditioned on its parent being unoccupied, 
(or equivalently, looking at it as the root of it subtree), if vertex $v$ is labeled `large' then 
there are maximum independent sets on its subtree which do not contain $v$ (\ie,  there are configurations in which $v$ is unoccupied), and $a_v >0$,  whereas if $v$ is labeled `small' then \emph{every} maximum independent set contains $v$ (\ie, $v$ is occupied in all configurations) and $a_v=0$.

Now consider a `large' node $v$ with two or more  `small' children. Looking at the recurrence~\eqref{eq:bp}, and the rules for arithmetic in the ring $\mathbb{R}[\lambda^{-1}]/(\lambda^{-2})$, we see that regardless of the non-occupation probabilities of all the other children of $v$, $a_v =1$.

In other words, if $v$ has two or more children that are
probably occupied, then $v$ is probably empty, regardless of what
other children it has.  We say in this situation that $a_v$ is known.
More generally, we say that  for `large' $v$, $a_v$ is known whenever it is determined 
by a finite subtree of $v$'s descendants.  In particular, known $a_v$s are rational.
For technical reasons, we will not say $a_v$ is known for all $v$ that are `small', but rather only those $v$ all of whose children are known

Let $\k_L$ and $\k_S$ denote the probability that $a_v$ is large
and known, or small and known, respectively.  If $a_v$ is large, it is
known either if it has two or more small children, or if all its children are
known and exactly one of them is small. If $a_v$ is small, then it is known if and only if all its
children (which are large) are known.  This gives us the equations
\begin{align}
\k_L &= 1-(1+d p_S) \,\e^{-d p_S} + d \k_S \,\e^{-d
  p_S} \,\e^{-d (p_L - \k_L)} 
\\ 
\k_S &= \e^{-d p_S}
\,\e^{-d (p_L-\k_L)} \, . 
\end{align}
Simplifying and combining with~\eqref{eq:ps} gives the
relations
\begin{align}
p_L - \k_L &= d (p_S^2 - \k_S^2) \label{eq:klarge}
\\ 
\k_S &= p_S  \,\e^{-d(p_L-\k_L)} \label{eq:ksmall}\, .
\end{align}
Rearranging terms and once again using~\eqref{eq:ps} we see that 
\[
\k_L = 1 - (1+d p_S) p_S + d
\e^{-2d(1-\k_L)} 
\]
so that $\k_L$ is a fixed point the function 
\[
g(x) := 1-(1+W(d))\frac{W(d)}{d} + d\e^{-2d(1-\k_L)}\, .
\]

The system of equations~\eqref{eq:klarge} and \eqref{eq:ksmall} always has 
$(\k_L, \k_S) = (p_L, p_S)$ as one solution.
Additionally, when $d$ is sufficiently large, there is a second solution
where $\k_L < p_L$ and $\k_S < p_S$, 
corresponding to the fact that for large enough $d<\e$ there are graphs for which even though the root $v$ is labeled ``large", the actual value of $a_v$ is not determined by any finite subtree of the Poisson tree.  
The threshold where these roots appear is
the $d$ such that
\[
g'(p_L) = 1 = 2 d^2 p_S^2 \, ,
\]
which with~\eqref{eq:plarge-psmall} implies
\[
d = \frac{\e^{1/\sqrt{2}}}{\sqrt{2}} = 1.434... \, .
\]

\section{Finite $\lambda$ case: Lower bound}\label{sec:finlam-SSM-lb}

In this section we will derive a lower bound on the SSM threshold for the 
Poisson tree. This proves part~\ref{part:SSM-1/d} of Theorem~\ref{thm:summary-finite}.

By Remark~\ref{rmk:SSM=WSM-forall-subgraphs}, in order to show SSM, it suffices to show WSM for any subtree of $\Poi(d)$.

Let $T$ be a subtree of $\Poi(d)$ and let $\rootvtx$ be the root of $T$. 
For $R>0$, let $T_R$ denote the truncation of $T$ to depth $R$, and let 
$\bd T_R$ denote the boundary of $T_R$, \ie, the vertices of $T$ at depth $R$.
We want to study the influence 
of the non-occupation probability values at $\bd T_R$  (set adversarially) on 
the value of $a_{\rootvtx}$. 
For notational convenience we will require the adversary to set the values 
at $\bd T_R$ from $[ \frac{1}{1+\lambda}, 1 ]$. Since the 
range of the function $x \mapsto \frac{1}{1 + \lambda P x}$ on $[0,1]$ is 
contained in $[ \frac{1}{1+\lambda}, 1]$ when $0< P \le 1$, this 
corresponds to allowing the adversary to set values in $[0,1]$ on $\bd T_{R+1}$.

Recall, from Proposition~\ref{prop:TFAE-WSM}, that to show WSM for $T$, it 
suffices to show that there is a well-defined non-occupation probability 
$a_{\rootvtx}^*$ at the root $\rootvtx$ of $T$. This, in turn,  would follow if
the non-occupation probabilities induced at $\rootvtx$ by setting the vertices 
in $\bd T_{R+1}$ to all zeroes or all ones converged to the same value as 
$R\to \infty$.

Let $w$ be a vertex of $\bd T_R$. Suppose the values at all the other vertices 
in $\bd T_R$ are fixed, and only the value $a_w$ at $w$ is varied. Let 
$a_{\rootvtx}^{[a_w=1/(1+\lambda)]}$ and $a_{\rootvtx}^{[a_w=1]}$ be the values of $a_{\rootvtx}$ 
when $a_w$ is set to $\frac{1}{1+\lambda}$
or 1 respectively. Then by the mean value theorem, 
\[
\left|a_{\rootvtx}^{[a_w=1/(1+\lambda)]} - a_{\rootvtx}^{[a_w=1]}\right|  \le \max_{a_w \in \left[ 
\frac{1}{1+\lambda},1 \right]} 
\left| \frac{\partial a_{\rootvtx}}{\partial a_w}\right|
\]
Now, if $a_{\rootvtx}^{\mathbf{0}}$ and $a_{\rootvtx}^{\mathbf{1}}$ are the values of $a_{\rootvtx}$ when the 
vertices at depth $R+1$ have been set respectively to all zeroes or all ones
(\ie, the vertices in $\bd T_R$ set to all ones or all $\frac{1}{1+\lambda}$)
then by varying 
the values at the boundary vertices one at a time and applying the triangle 
inequality, we see that 
\begin{equation}\label{eq:av0av1bd}
\left| a_{\rootvtx}^{\mathbf{0}}-a_{\rootvtx}^{\mathbf{1}} \right|  
\le \sum_{w \in \bd T_R} 
\max_{a_w \in  \left[ \frac{1}{1+\lambda},1 \right] } \left| 
\frac{\partial a_{\rootvtx}}{\partial a_w}\right|
\end{equation}

Fix $w \in \bd T_R$ and let ${\rootvtx} = w_0 , w_1, \dots w_{R-1}, w_R= w$
be the path from the root to $w$.  Let $a_i = a_{w_i}$ and let 
$P_i = \prod_{x} a_x$ where the product is taken over all the children 
$x$  (if any) of $w_i$ other than $w_{i+1}$. 
Then for all $i$, 
\[
a_i = \frac{1}{1+\lambda a_{i+1} P_i}
\]
Note that $a_i \ge \frac{1}{1+\lambda} >0$ for al $i$.
Differentiating $a_i$ with respect to $a_{i+1}$, with some algebraic 
manipulations, we have
\[
\frac{\partial a_i}{\partial a_{i+1}} = \frac{-\lambda P_i}
{(1+\lambda a_{i+1}P_i)^2} =
\frac{-a_i(1-a_i)}{a_{i+1}}
\]
Repeatedly applying the chain rule, we see that
\[
\frac{\partial a_{\rootvtx}}{\partial a_w} = \frac{\partial a_0}{\partial a_{R}} 
= \prod_{i=0}^{R-1} \frac{\partial a_i}{\partial a_{i+1}}
= \prod_{i=0}^{R-1} \frac{-a_i(1-a_i)}{a_{i+1}}
= (-1)^{R}  \frac{a_0}{a_R}\prod_{i=0}^{R-1}(1-a_i)
\]
Since 
$\frac{1}{1+\lambda} < a_i \le 1$,
\begin{equation}\label{eq:leafpartial}
\left|\frac{\partial a_{\rootvtx}}{\partial a_w}\right|   \le \frac{a_0}{a_R}
\prod_{i=0}^{R-1}(1-a_i) \le (1+\lambda)\prod_{i=0}^{R-1}(1-a_i).
\end{equation}

Note that $a_i \ge\frac{1}{1+\lambda a_{i+1}}$, so 
that $(1-a_i)(1-a_{i+1}) \le \frac{\lambda a_{i+1}(1-a_{i+1})}{1+\lambda a_{i+1}}$.
To bound the partial derivative, we want to maximze this subject to the constraint that $a_{i+1}\ge \frac{1}{1+\lambda}$.

Consider the function $x \mapsto  \frac{\lambda x(1-x)}{1+\lambda x}$ on the 
interval $[\frac{1}{1+\lambda}, 1]$. Differentiating, 
we see that  when $\lambda \ge \frac{1+\sqrt{5}}{2}$, it is maximized at 
$\frac{\sqrt{1+\lambda} -1}{\lambda}$ and that 
the maximum value is $1-\frac{2}{\lambda}\left(\sqrt{1+\lambda}-1\right)$. 
Thus $(1-a_i)(1-a_{i+1}) \le 1-\frac{2}{\lambda}\left(\sqrt{1+\lambda}-1\right)$.
 Applying 
this to consecutive pairs in $\prod_{i=0}^{R-1}(1-a_i)$, we have, for even $R$
 \begin{equation}\label{eq:fancybd}
\left|\frac{\partial a_{\rootvtx}}{\partial a_w}\right|  \le 
(1+\lambda)\prod_{i=0}^{R-1}(1-a_i) \le
(1+\lambda)
\left( 1-\frac{2}{\lambda}\left(\sqrt{1+\lambda}-1\right)\right)^{R/2}
\end{equation}

On the other hand, if $\lambda < \frac{1+\sqrt{5}}{2}$, then the derivative of
$\frac{\lambda x(1-x)}{1+\lambda x}$ is never zero in 
$[\frac{1}{1+\lambda}, 1]$, and the function is maximized at 
$\frac{1}{1+\lambda}$. Thus 
$(1-a_i)(1-a_{i+1}) \le \frac{\lambda^2}{(1+\lambda)(1+2\lambda)} $, and once again, applying this to consecutive pairs, for even $R$,
\begin{equation}\label{eq:fancybd2}
\left|\frac{\partial a_{\rootvtx}}{\partial a_w}\right|  \le (1+\lambda)
\left( \frac{\lambda^2}{(1+\lambda)(1+2\lambda)}\right)^{R/2}
\end{equation}

Let us now re-examine \eqref{eq:av0av1bd}.  We have 
\begin{equation}\label{eq:av0av1bd2}
\left| a_{\rootvtx}^{\mathbf{0}}-a_{\rootvtx}^{\mathbf{1}} \right|  \le \sum_{w \in \bd T_R} 
\max_{a_w \in \left[\frac{1}{1+\lambda},1\right]} \left| \frac{\partial a_{\rootvtx}}{\partial a_w}\right|
\le |\bd T_R| B_{\lambda, R}
\end{equation}
where $B_{\lambda, R}$ is an upper bound on 
$\left| \frac{\partial a_{\rootvtx}}{\partial a_w}\right|$.

Since $T$ is a subtree of a $\Poi(d)$ tree, it follows from Proposition~\ref{prop:R2dR} that, almost surely, for all sufficiently large $R$
\begin{equation}\label{eq:bdsize}
|\bd T_R| \le R^2 d^{R}.
\end{equation}

If $\lambda \ge \frac{1+\sqrt{5}}{2}$ then substituting 
$B_{\lambda, R}=  (1+\lambda)
\left( 1-\frac{2}{\lambda}\left(\sqrt{1+\lambda}-1\right)\right)^{R/2}$ into 
\eqref{eq:av0av1bd2}, we have 
\[
\left| a_{\rootvtx}^{\mathbf{0}}-a_{\rootvtx}^{\mathbf{1}} \right|  \le R^2d^R 
(1+\lambda)
\left( 1-\frac{2}{\lambda}\left(\sqrt{1+\lambda}-1\right)\right)^{R/2}
= R^2 (1+\lambda)
\left[d^2\left( 1-\frac{2}{\lambda}\left(\sqrt{1+\lambda}-1\right)\right)
\right]^{R/2}
\]
which goes to 0 as $R \rightarrow \infty$ as long as 
$d^2\left( 1-\frac{2}{\lambda}\left(\sqrt{1+\lambda}-1\right)\right) <1$, \ie,
$\lambda < \frac{4d^2}{(d^2-1)^2}$. 

If $\lambda < \frac{1+\sqrt{5}}{2}$ then substituting 
$B_{\lambda, R}=  (1+\lambda)\lambda^R (1+\lambda)^{-R/2}(1+2\lambda)^{-R/2}$
into \eqref{eq:av0av1bd2}, we have 
\[
\left| a_{\rootvtx}^{\mathbf{0}}-a_{\rootvtx}^{\mathbf{1}} \right|  \le R^2d^R 
(1+\lambda)\frac{\lambda^R}{(1+\lambda)^{R/2}(1+2\lambda)^{R/2}}
= R^2 (1+\lambda)
\left[\frac{d^2\lambda^2}{(1+\lambda)(1+2\lambda)}\right]^{R/2}
\]
which  goes to 0 as $R \rightarrow \infty$ as long as 
$d^2\lambda^2 < (1+\lambda)(1+2\lambda)$, \ie, $\lambda<\frac{3+\sqrt{1+4d^2}}{2d^2-4}$.

The transition point, $\lambda= \frac{1+\sqrt{5}}{2}$ corresponds to $d=\sqrt{2+\sqrt{5}}$ which is approximately 2.058. 

Thus we have shown WSM for independent sets with fugacity 
$\lambda$ on any subtree $T$ of a $\Poi(d)$ tree, when  
\[
\lambda < \begin{cases}\frac{4d^2}{(d^2-1)^2} & \mbox{when } d<\sqrt{2+\sqrt5}\\
\frac{3+\sqrt{1+4d^2}}{2d^2-4} & \mbox{otherwise.}
\end{cases}
\] 
By Remark~\ref{rmk:SSM=WSM-forall-subgraphs} we have SSM for $\Poi(d)$ for $\lambda$ in the same range.

\section{Mixing for small $d$}

In this section we prove part~\ref{part:1.179} of
Theorem~\ref{thm:summary-finite}, which we now restate in an
equivalent form.

\begin{theorem}
\label{thm:dmorethanone}
For all $d < 1.179...$, the $\Poi(d)$ tree almost surely has weak spatial
mixing for all finite $\lambda > 0$.
\end{theorem}

\begin{proof}
Recall our formula for the influence of a leaf $w$ along the path $v=v_0, v_1, \ldots, v_R = w$:
\begin{equation}
\label{eq:influence}
\abs{ \frac{\partial \ln a_v}{\partial \ln a_w} } = \prod_{i=1}^{R-1} (1-a_i) \, . 
\end{equation}
We claim that the existence of this path tells us nothing about the other branches of the tree that do not survive to depth $R$.  In particular, the number of childless children of each $v_i$ for $0 \le i < R-1$ is independent, and Poisson-distributed with mean $\mu = d \e^{-d}$.

The presence of these small leaves gives us a better upper bound on $1-a_i$.  In particular, if $v_i$ has $c_i$ childless children, then 
\[
1-a_i 
\le 1 - \frac{1}{1+\lambda \left( \frac{1}{1+\lambda} \right)^{c_i} } 
= \frac{\lambda}{(1+\lambda)^{c_i} + \lambda} \, . 
\]
Thus $w$'s expected influence is at most
\begin{align*}
\Exp_{\{c_i\}} \prod_{i=0}^{R-2} \frac{\lambda}{(1+\lambda)^{c_i} + \lambda} 
= \left( \Exp_c \frac{\lambda}{(1+\lambda)^c + \lambda} \right)^{R-1} 
\le \left( \e^{-\mu} \,\frac{\lambda}{1+\lambda} + (1-\e^{-\mu}) \,\frac{\lambda}{1 + 2\lambda} \right)^{R-1} \, .  
\end{align*}
The expected total influence of all the leaves is this times $d^R$, which is exponentially small if
\[
\e^{-\mu} \,\frac{\lambda}{1+\lambda} + (1-\e^{-\mu}) \,\frac{\lambda}{1 + 2\lambda} < \frac{1}{d} \, . 
\]
The left-hand side is monotonically increasing with $\lambda$, so this inequality holds as long as 
\[
\frac{1+\e^{-\mu}}{2} < \frac{1}{d} \, . 
\]
Substituting $\mu = d \e^{-d}$, we find that this holds for all $d < 1.179$.
\end{proof}

We have made no attempt to optimize the constant in Theorem~\ref{thm:dmorethanone}.  

\begin{figure}
   \begin{center}
      \includegraphics[width=2.0in]{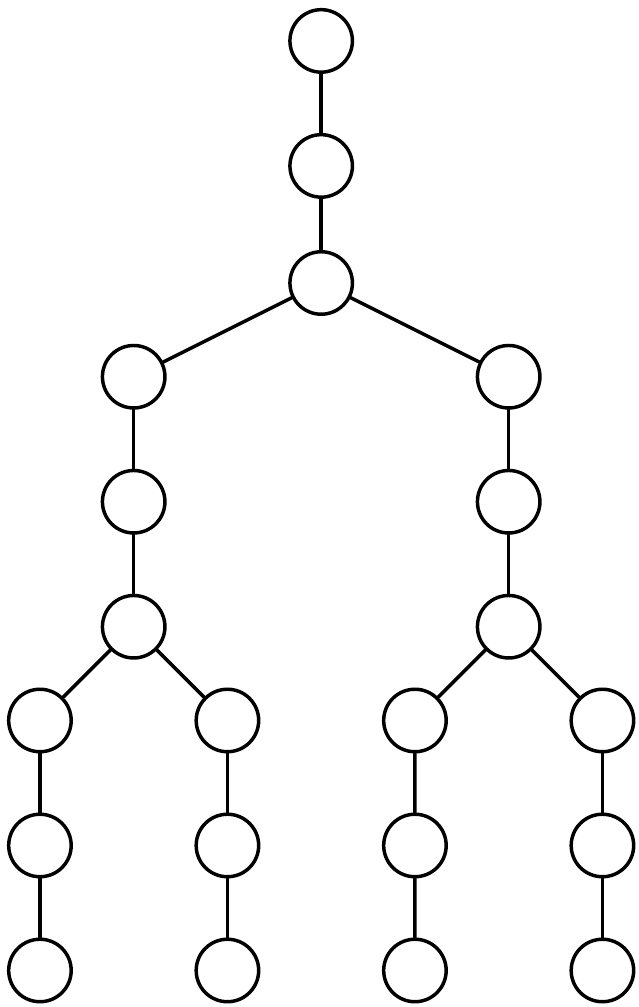}
      \caption{An example of a stretched binary tree with $c = 3$.}
      \label{fig:stretched}
   \end{center}
\end{figure}

\begin{remark} \label{rmk:noSSM} Note that for any $d > 1$, there is a $\lambda$ for which $\Poi(d)$ tree lacks strong spatial mixing.  The reason (as pointed out to us by Allan Sly) is that it possesses, with positive probability, subgraphs that are ``stretched'' versions of the infinite binary tree, which branch every $c$ generations for some constant $c$.  See Figure~\ref{fig:stretched}. Such trees lack weak spatial mixing for sufficiently large $\lambda$, since if 
\[
f_1(a) = \frac{1}{1+\lambda a} 
\quad \text{and} \quad
f_2(a) = \frac{1}{1+\lambda a^2} \, ,
\]
the function 
\[
\underbrace{f_1(f_1(\cdots(f_1}_{\textrm{$c$ times}}(f_2(a)))))
\]
has a stable period-$2$ orbit for sufficiently large $\lambda$.
\end{remark}

\section{Non-mixing just above the threshold} \label{sec:no-WSM-asymp}

In this section we will prove that, for sufficiently large but constant $d$, 
the $\Poi(d)$ tree lacks spatial mixing just above the threshold for 
$d$-regular trees.  First note that the latter is
\[
\frac{d^d}{(d-1)^{d+1}} = \frac{\e}{d} + O(1/d)^2 \, . 
\]

Note that for $z \in [-1, 1]$,
\begin{equation}\label{eq:taylor}
1-z \le \frac{1}{1+z} \le 1-z +z^2
\end{equation}
Let $v$ be a vertex at level $L-1$.  By \eqref{eq:taylor} and the definition 
of $a_v$ \eqref{eq:bp}, we have
\[
1-\lambda \prod_{w}a_w \le a_v \le 1 -\lambda \prod_{w}a_w + \lambda^2 
\prod_{w}a_w^2
\]
where the product is over the children $w$ of $v$, which are at level $L$. 
Taking expectations, we have 
\begin{equation}\label{eq:Expav}
1- \lambda\Exp\left[\prod_{w}a_w\right] \le \Exp a_v \le 
1- \lambda\Exp\left[\prod_{w}a_w\right]+ \lambda^2\Exp\left[\prod_{w}a_w^2\right] 
\end{equation}
Let $a_L$ denote the non-occupation probability of a 
\emph{generic} vertex at level $L$, in a Poisson tree truncated at depth $R$. 
(Note that these are independent and identically distributed.) 
Let $K\sim \Poi(d)$ denote the number of children of vertex $v$, and let 
$a_1, a_2, \dots a_K$ denote the non-occupation 
probabilities of these children. Then the $a_i$s are independent of each other 
and $K$ and each has expectation $\Exp a_{L}$.  So 
\begin{align*}
\Exp\left[\prod_{w}a_w\right] &= \Exp\left[ \ExpCond{\prod_{i=1}^{K}a_i}{K}
\right]\\
&=\Exp\left[\prod_{i=1}^{K} \ExpCond{a_i}{K} \right] \\
&=\Exp\left[\left[\Exp a_{L}\right]^K\right] \\
&=\sum_{k=0}^\infty  \frac{\e^{-d}d^k}{k!}[\Exp a_{L}]^k\\
&=\e^{-d(1-\Exp a_{L})}
\end{align*} 
Similarly, 
\[
\Exp\left[\prod_{w}a_w^2\right] = \e^{-d(1-\Exp a_L^2)}
\]
Substituting into~\eqref{eq:Expav}, we have
\begin{equation}
\label{eq:phiq-bound0}
1- \lambda \e^{-d(1-\Exp a_{L})} 
\le \Exp a_{L-1} 
\le 1- \lambda \e^{-d(1-\Exp a_{L})} + \lambda^2 \,\e^{-d(1-\Exp a_{L}^2)} \, .
\end{equation}
If we define
\begin{equation}
\label{eq:phiq}
\phi_q(z) = 1 - \lambda \e^{-d(1-z)} + q \lambda^2 \, ,
\end{equation}
we can rewrite~\eqref{eq:phiq-bound0}
\begin{equation}
\label{eq:phiq-bound}
\phi_0(\Exp a_{L-1}) 
\le \Exp a_L 
\le \phi_q(\Exp a_{L-1}) \, , 
\end{equation}
where 
\[
q  = \e^{-d(1-\Exp a_{L}^2)} \in [0,1] \, . 
\]

The following lemma shows that for $\lambda$ just above $\e/d$, even if an adversary controls the second moment $\Exp a_{L}^2$ and hence the coefficient $q$ of the quadratic term, this function oscillates between two disjoint intervals.  It follows that the expected occupation probability at the root alternates between high and low values based on the parity of the depth of the tree, implying a lack of spatial mixing.

\begin{lemma}\label{lem:oscillation}
For fixed $\lambda$, $d$, and $q \in [0,1]$, let $\phi_q(z)$ be defined as in~\eqref{eq:phiq}. 
Let $\lambda = c/d$ where $c > \e$ is a constant.  Then there are constants 
$d^*$, $b_1$, and $b_2$ such that, for all $d > d^*$ and all $q \in [0,1]$, 
\begin{align*}
\forall z > 1-b_1/d :&  \, \phi_q(z) < 1-b_2/d \\ 
\forall z < 1-b_2/d :&  \, \phi_q(z) > 1-b_1/d \, . 
\end{align*} 
and where $b_1 < b_2$.
\end{lemma}
\begin{proof}
Since $\phi_0$ is monotonically decreasing, it has a unique fixed point 
$z_0 = \phi_0(z_0)$, namely
\[
z_0 = 1-\frac{b_0}{d} \;\text{ where }\; b_0 = W(\lambda d) = W(c) \, . 
\] 
Here $W(x)$ is Lambert's function, i.e., the unique positive root $y$ 
of $y\e^y = x$.  We have $f'_0(x_0) = -W(d \lambda)$.  If $c > \e$ then 
$W(c) > 1$, making this fixed point unstable.

To focus on $\phi_q$'s behavior near $z_0$ we change variables, setting 
$z=z_0 + \delta/d$.  Then applying $\phi_q$ to $z$ is equivalent to applying 
$\psi_q(\delta)$ to $\delta$, where
\begin{align*}
\psi_q(\delta) = d \cdot (\phi_q(z_0+\delta/d) - z_0) 
=  - (\e^\delta - 1) \,W(\lambda d) + q \lambda^2 d 
=  - (\e^\delta - 1) \,W(c) + \frac{q c^2}{d} \, .
\end{align*}
Since $g'_0(0) = -W(c)$ and $\psi_0$ is analytic, for any constant $1 < A < W(c)$, there is a constant $\tilde{\delta} > 0$ such that 
\[
\forall \delta \in [-\tilde{\delta}, \tilde{\delta}] : \, g'_0(\delta) < -A \, . 
\]
Therefore, for any $\delta^* < \tilde{\delta}$ we have
\begin{align*}
\forall \delta > \delta^*: \,\psi_0(\delta) < -A \delta^*  \mbox{\,\, and \,\,}
\forall \delta < -\delta^*: \,\psi_0(\delta) > A \delta^* \, . 
\end{align*} 
Choose such an $A$ and such a $\delta^*$ with $\delta^* < b_0$.  Finally, since $\psi_0(\delta) \le \psi_q(\delta) \le \psi_0(\delta) + c^2 / d$, if $d > d^*$ is sufficiently large so that 
\[
\frac{c^2}{d} < (A-1) \delta^* \, , 
\]
the proof is completed by setting $b_1 = b_0 - \delta^*$ and $b_2 = b_0 + \delta^*$.
\end{proof}

\section{Asymptotically Optimal Lower Bound}\label{sec:asymp-lb-abbr}

We saw in Section~\ref{sec:no-WSM-asymp} that asymptotically, for large $d$ the 
$\Poi(d)$ tree does not have weak spatial mixing for $\lambda$ just 
above $\e/d$, which is the asymptotic threshold for WSM (and SSM) for the 
$d$-regular tree. We will now show  that  \emph{below} $\e/d$ 
the $\Poi(d)$ tree almost certainly does have weak spatial mixing. 
Specifically we will prove the following result, which is equivalent
to part~\ref{part:WSM} of Theorem~\ref{thm:summary-finite}.
\begin{theorem}\label{thm:asymplb-main}
For all $\gamma \in (0, 1)$, for all sufficiently large $d$,
the $\Poi(d)$ tree with activity $\lambda = (1 - \gamma) \e / d$
exhibits weak spatial mixing with probability $1$.
\end{theorem}

The proof is fairly involved, and we begin by presenting a summary 
of the main ideas involved.

\begin{proof}[Proof Sketch]
To show WSM we need to show that there is a well defined non-occupation 
probability $a_{\rootvtx}^*$ at the root, \ie, that the sequences 
$a_{\rootvtx, R}^{\mathbf{0}}$ and $a_{\rootvtx, R}^{\mathbf{1}}$ converge to a common 
limit. 
As in Section~\ref{sec:finlam-SSM-lb} we bound 
$|a_{\rootvtx, R}^{\mathbf{0}} -a_{\rootvtx, R}^{\mathbf{1}}|$ by the sum of the 
absolute values of the partial 
derivatives $\partial a_{\rootvtx}/\partial a_w$ where $w$ is a vertex at depth 
$R$. We know that there are almost surely  at most $R^2d^R$ such
vertices, for all sufficiently large $R$. 
The improvement in this argument comes from proving a better upper bound on 
$\prod_v (1-a_v)$ which controls the size of 
$|\partial a_{\rootvtx}/\partial a_w|$. Here, the product is taken over all 
vertices $v$ on the path from $\rootvtx$ to $w$.
The main idea is that when $d$ is very large, most of the vertices on the path
from $\rootvtx$ to $w$ are ``good'' in the sense that they and all their 
descendants to some depth $h$ have degrees very close to $d$. In other
words, each such vertex $v$ is the root of a nearly regular $d$-ary subtree of depth $h$. For large enough $h$, this means that $a_v$ is very close to the fixed point $a^*$ of 
the function $f_d(x)= (1+\lambda x^d)^{-1}$, which exists since
$\lambda$ is less than the regular $d$-ary threshold. Thus for each good
vertex $v$, $(1-a_v) < c/d$ for some small $c<1$ and it only remains
to show that there are almost surely enough good vertices that, for
all sufficiently large $R$,  
the bound $\prod_v (1-a_v)$ for each path to depth $R$ beats the $R^2d^R$ such paths.
\end{proof}

We devote the rest of this section to making the above argument rigorous.

\begin{remark} Unlike the proof in Section~\ref{sec:finlam-SSM-lb}, this proof does 
\emph{not} show strong spatial mixing. Passing to a subtree can destroy the property 
that most vertices have nearly $d$-ary subtrees to some 
depth (or even that they have degree close to $d$). Given the results
in Section~\ref{sec:finlam-SSM-lb}, it is an open question whether SSM
holds with high probability for $\lambda$ between $1/d$ and $\eee/d$.
\end{remark} 

The proof of Theorem~\ref{thm:asymplb-main} rests heavily of the fact that 
most of the vertices in the $\Poi(d)$ tree are roots of subtrees (to some 
depth) that are almost $d$-ary. In order to make precise what we mean by 
``almost $d$-ary'', we will first need some definitions.

\begin{definition} An $(a,b)$-tree is an infinite rooted tree in which every 
vertex at an even depth has $a$ children and every node at an odd depth has 
$b$ children. A truncated $(a,b)$-tree is the truncation of an $(a,b)$-tree to 
some finite depth $R$.
\end{definition}

\begin{definition} Let $0 < \Delta_1 \le \Delta_2$. A rooted tree $T$ is 
$[\Delta_1, \Delta_2]$-regular if the number of children of every vertex is in 
 $[\Delta_1, \Delta_2]$. 
\end{definition}

By an almost $d$-ary tree, we will mean a 
$[(1-\eps)d, (1+\eps)d]$-regular tree. In what follows we will show that such 
a tree behaves like a $d$-ary tree, in that if the tree is sufficiently deep, 
then for almost the same range of $\lambda$ as for the $d$-ary tree, the 
non-occupation probabilities converge to well defined value at the root.

Our next result gives us a way to find a $(\Delta_1, \Delta_2)$ tree 
and a $(\Delta_2, \Delta_1)$ tree ``near'' any $[\Delta_1,
\Delta_2]$-regular tree.  See Figure~\ref{fig:pg-overall} for
illustrations.

\begin{figure}[!ht]
  \centering
  \begin{subfigure}[b]{0.90\textwidth}
    \centering
    \includegraphics[width=0.33\textwidth]{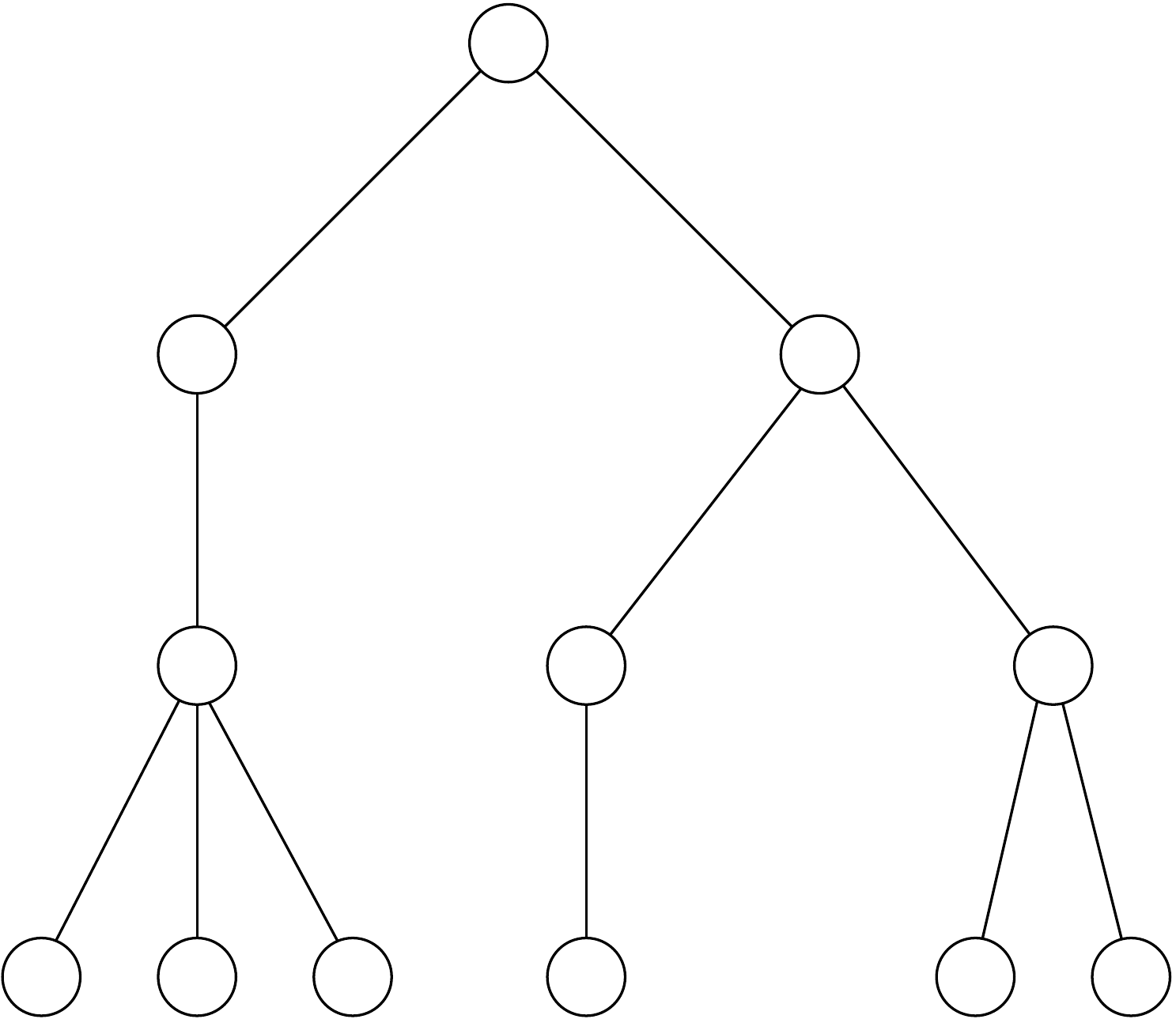}
    \caption{Original tree}
    \label{fig:pg-orig}
  \end{subfigure}
  \begin{subfigure}[b]{0.48\textwidth}
   \centering
  \includegraphics[width=0.9\textwidth]{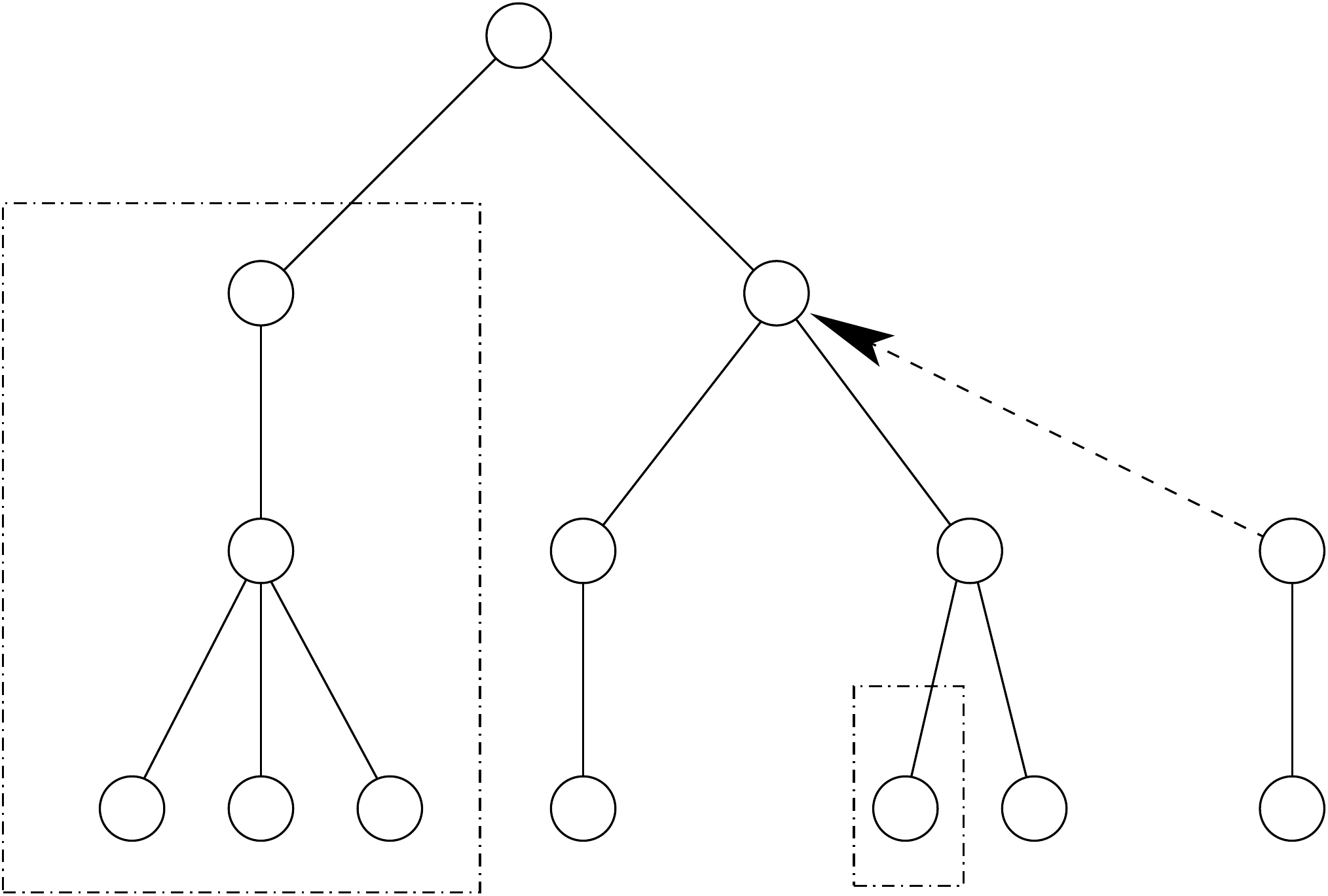}
\caption{Two branches pruned, one grafted.}
  \label{fig:pg-131-during}
  \end{subfigure}
 \begin{subfigure}[b]{0.48\textwidth}
   \centering
  \includegraphics[width= 0.65\textwidth]{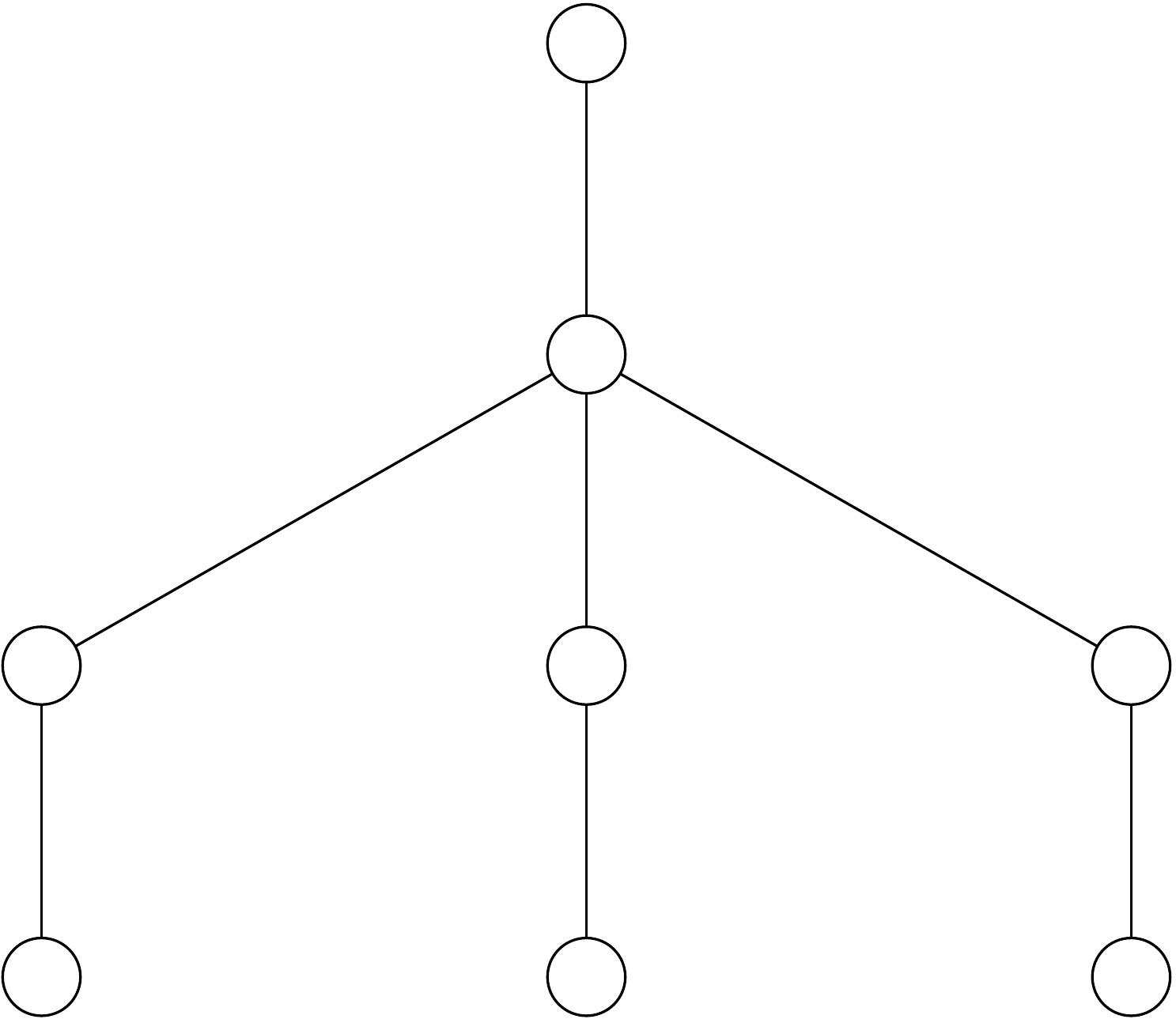}
  \caption{The resulting $(1,3)$-tree.}
  \label{fig:pg-131}
  \end{subfigure} 
 \begin{subfigure}[b]{0.49\textwidth}
    \centering
    \includegraphics[width=0.99\textwidth]{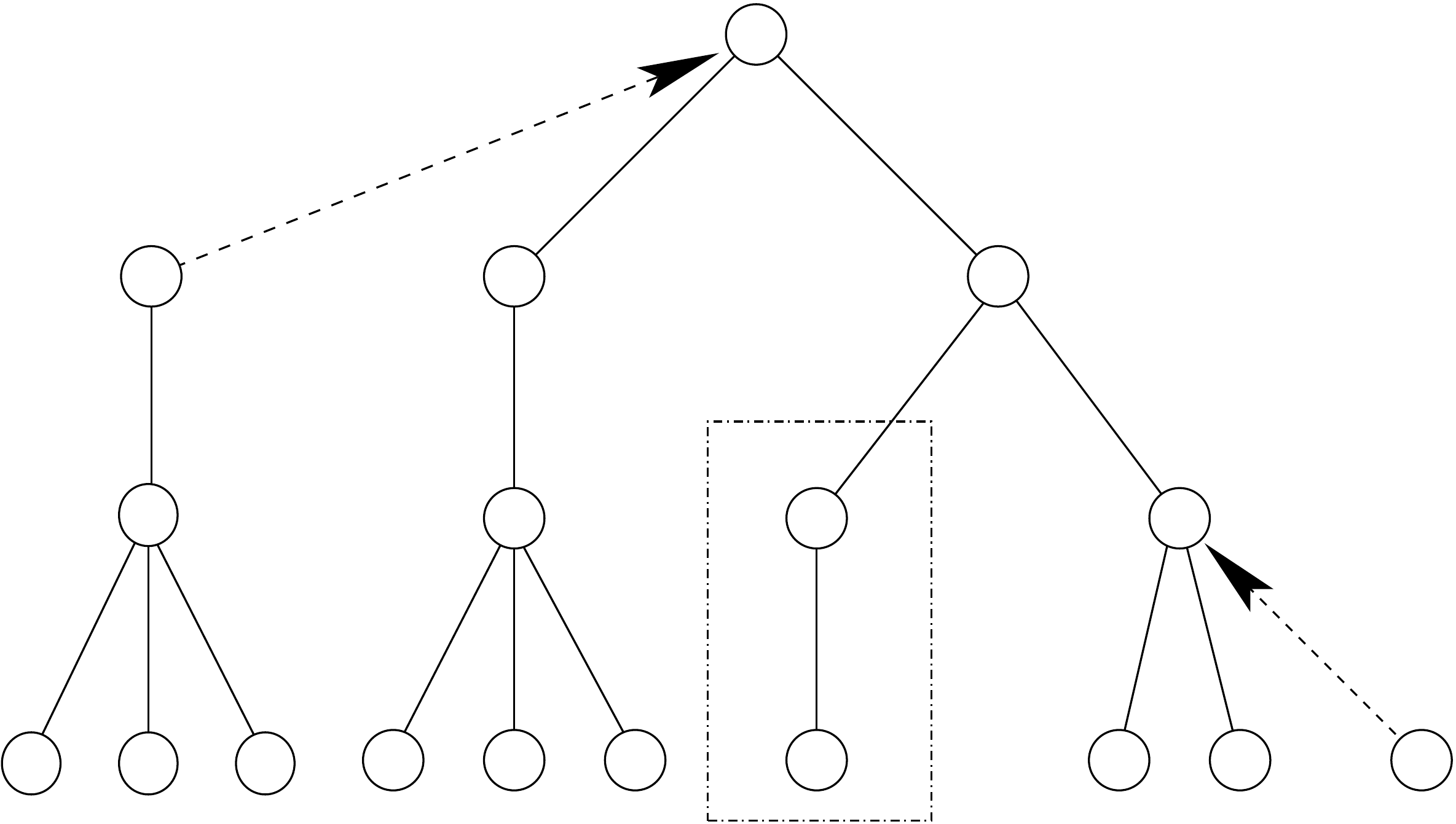}
\caption{Opposite parity of  levels.}
   \label{fig:pg-313-during}
 \end{subfigure}
  \begin{subfigure}[b]{0.49\textwidth}
   \centering
  \includegraphics[width=0.8\textwidth]{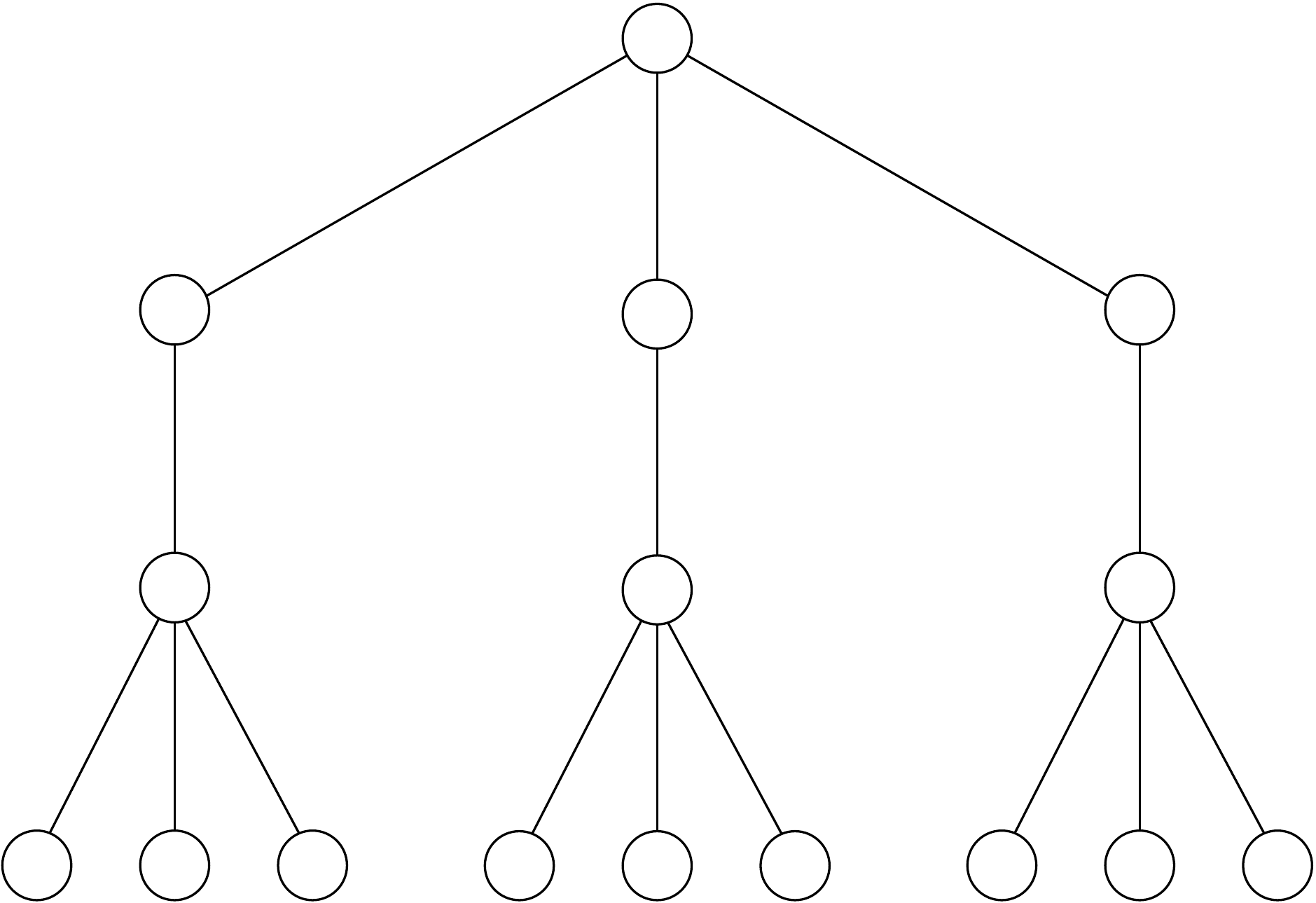}
  \caption{The resulting $(3,1)$-tree.}
  \label{fig:pg-313}
  \end{subfigure} 
 \caption{Applying Lemma~\ref{lem:prune-graft}.
   Old subtrees are pruned and new ones grafted on, on alternating levels.}
 \label{fig:pg-overall}
\end{figure}

\begin{lemma}[Pruning/Grafting] \label{lem:prune-graft} Let $T$ be a $[\Delta_1, \Delta_2]$-regular 
tree with root $v$ and depth $R$. Then 
\begin{enumerate}
\item $T$ can be transformed into a truncated $(\Delta_1, \Delta_2)$-tree $T'$
of depth $R$, rooted at $v$,
by pruning (removing children along with their entire subtrees) at even levels 
and grafting (adding children together with an appropriate subtree) at odd 
levels. 
\item $T$ can be transformed into a truncated $(\Delta_2, \Delta_1)$-tree 
$T''$ of depth $R$,  rooted at $v$, by grafting at even levels and pruning 
at odd levels. 
\end{enumerate}
Let $a_v$, $a_v'$ and $a_v''$ denote the non-occupation probabilities at the 
root in  $T$, $T'$ and $T''$ respectively, when all their leaves are set to 
the same value
$a_0 \in [0,1]$ . Then 
\[
a_v' \le a_v \le a_v''
\]
\end{lemma}

\begin{proof} By induction on depth of $T$. 
\end{proof}

Recalling that 
\[
f_d(a) = \frac{1}{1+\lambda a^d} \, , 
\]
we wish to prove, for certain values of $\lambda$, that iterating $f_{\Delta_1} \circ f_{\Delta_2}$ causes $a_v$ to converge to a unique fixed point.  The following two lemmas establish the existence and uniqueness of this fixed point, and bound its location.

\begin{lemma}
\label{lem:d1d2contraction}
Let $\Delta_1, \Delta_2 \ge 2$, 
and let 
\begin{equation}
\label{eq:lambdad1d2}
\lambda(\Delta_1, \Delta_2) 
= \Delta_2^{\Delta_1}  \left( \frac{\Delta_1 + 1}{\Delta_1 \Delta_2 - 1} \right)^{\!\Delta_1 + 1} \, .
\end{equation}
For any $\lambda < \lambda(\Delta_1, \Delta_2)$, there is a unique fixed point $a^*$ such that $(f_{\Delta_1} \circ f_{\Delta_2})(a^*) = a^*$.  Moreover, 
there is a constant $c < 1$ such that 
\[
\abs{ (f_{\Delta_1} \circ f_{\Delta_2})^t (a_0) - a^* } 
\le c^{t-1} \ln (\lambda+1) \, . 
\]
Moreover, 
\[
c \le \frac{1}{2} \left( 1 + \frac{\lambda}{\lambda(\Delta_1,\Delta_2)} \right) \, .
\]
\end{lemma}

\begin{proof}
We will begin by changing variables.  First define $y=\ln a$, in which case $y \in (-\infty,0]$ and 
\[
g_d(y) = - \ln (1+\lambda \e^{dy}) \, .
\]
Note that $g_{\Delta_1} \circ g_{\Delta_2}$ is monotonically increasing.  We will show that, for any $\lambda < \lambda(\Delta_1, \Delta_2)$, there is a constant $c < 1$ such that
\begin{equation}
\label{eq:contract}
\ddy \,g_{\Delta_1} (g_{\Delta_2} (y)) 
= g'_{\Delta_1}(g_{\Delta_2} (y)) \,g'_{\Delta_2}(y) 
\le c 
\text{ for all }
y \le 0 \, . 
\end{equation}  
This implies that the fixed point $y^*= (g_{\Delta_1} \circ g_{\Delta_2})(y^*) = \ln a^*$ is unique, and that we approach it exponentially quickly as we iterate $g_{\Delta_1} \circ g_{\Delta_2}$.  Rather than finding $c$ as a function of $\lambda$, it is analytically simpler to find a $\lambda$ such that~\eqref{eq:contract} holds for a given $c$, and then showing that this $\lambda$ coincides with $\lambda(\Delta_1,\Delta_2)$ when $c=1$.

It is convenient to do one more change of variables, from $y$ to $g^{-1}_{\Delta_2}(y)$ (which is well-defined since $g_d$ is monotonic).  Thus we can focus on
\[
h(y) 
= g'_{\Delta_1}(y) \,g'_{\Delta_2}(g^{-1}_{\Delta_2}(y)) 
= \Delta_1 \Delta_2 \,\e^{\Delta_1 y} (1-\e^y) \,\frac{\lambda}{1+\lambda \,\e^{\Delta_1 y}}
\]
We will find a $\lambda$ such that $h(y) \le c$ for all $y \le 0$.  For any fixed $y$, $h(y)$ is a monotonically increasing function of $\lambda$.  
Moreover, we can find the $\lambda$ where $h(y)=c$, namely 
\[
\lambda_c(y) = \frac{c \e^{-\Delta_1 y}}{\Delta_1 \Delta_2 (1-\e^y)-c} \, , 
\]
where we note that if $\Delta_1 \Delta_2 (1-\e^y) < c$ then $h(y) < c$ for all $\lambda > 0$.  Taking derivatives, we find that $\lambda_c(y)$ is minimized at 
\[
y_{\min} = \ln \frac{\Delta_1 \Delta_2-c}{(1 + \Delta_1) \Delta_2} \, ,
\]
where
\begin{equation}\label{eq:lamc}
\lambda_c 
= \lambda_c(y_{\min}) 
= c \Delta_2^{\Delta_1} \left( \frac{\Delta_1+1}{\Delta_1 \Delta_2 - c} \right)^{\Delta_1+1} \, . 
\end{equation}
Thus if $\lambda \le \lambda_c$, we have $h(y) \le c$ for all $y \le 0$.  

Now note that $\lambda_c$ is a strictly increasing function of $c$, and that it ranges from $0$ to $\lambda(\Delta_1,\Delta_2)$ as $c$ goes from $0$ to $1$.  Thus for any $0 \le \lambda < \lambda(\Delta_1,\Delta_2)$ there is a $c = c(\lambda)  < 1$ such that $\lambda = \lambda_c$, and~\eqref{eq:contract} holds.  
Specifically, an easy calculation shows that $\mathrm{d}^2 \lambda_c / \dc^2 \ge 0$ for $0 < c < 1$, and that 
\[
\left. \frac{1}{\lambda(\Delta_1,\Delta_2)} \frac{\dlambda}{\dc} \right|_{c=1} = \frac{\Delta_1(\Delta_2+1)}{\Delta_1 \Delta_2 - 1} \le 2 \, . 
\]
(Indeed, this derivative is $1+O(1/\Delta_2)$.)  Therefore, 
\[
\lambda_c \ge \lambda(\Delta_1,\Delta_2) \big( 1-2(1-c) \big) \, ,
\]
and so
\[
c \le \frac{1}{2} \left( 1 + \frac{\lambda}{\lambda(\Delta_1,\Delta_2)} \right) \, .
\]

To complete the proof, each time we iterate $g_{\Delta_1} \circ g_{\Delta_2}$, any interval shrinks by a factor of $c$.  Since $g_{\Delta_1} \circ g_{\Delta_2}$ maps $(-\infty,0]$ into $(-\ln (\lambda+1),0]$, the width of any interval after $t$ iterations is at most $c^{t-1} \ln (\lambda+1)$.  The same bound holds when we change variables back to $a = \e^y$, since $\mathrm{d} \e^y / \mathrm{d}y \le 1$ for all $y \le 0$. 
\end{proof}

Note that when $\Delta_1 = \Delta_2 = \Delta$, the  value of $\lambda$ defined in Lemma~\ref{lem:d1d2contraction} becomes the known value for the $\Delta$-regular tree,
\[
\lambda(\Delta, \Delta) = \Delta^\Delta  \left( \frac{\Delta + 1}{\Delta^2 - 1} \right)^{\!\Delta + 1} = \frac{\Delta^\Delta}{(\Delta-1)^{\Delta+1}} \, .
\]
We will also use the following lower bound, 
\begin{equation}
\label{eq:lambdad1d2simple}
\lambda(\Delta_1, \Delta_2) 
\ge \Delta_2^{\Delta_1}  \left( \frac{\Delta_1 + 1}{\Delta_1 \Delta_2} \right)^{\!\Delta_1 + 1} 
= \frac{1}{\Delta_2} \left(1+\frac{1}{\Delta_1} \right)^{\!\Delta_1+1} 
\ge \frac{\e}{\Delta_2} \, .
\end{equation}

\begin{lemma}\label{lem:fp-lowerbd}
Let $\gamma \in (0,1)$ and let $\lambda = \frac{(1-\gamma)\e}{d}$.  
Let $\eps = \gamma^2/4$.  There is a constant 
$d_0=d_0(\gamma)$  such that 
for all $d > d_0$, the fixed point $a^*$ 
of $f_{(1-\eps)d}\circ f_{(1+\eps)d}$
is at least $1-\frac{1}{(1+\eps)d}$.
\end{lemma}

\begin{proof}
As before, we change variables to $y = \ln a$, and consider the fixed point of $g_{(1-\eps)d} \circ g_{(1+\eps)d}$ where $g_d(y) = - \ln (1+\lambda \e^{dy})$.  First, we show the conditions of Lemma~\ref{lem:d1d2contraction} are met.  Recall the definition of $\lambda(\Delta_1,\Delta_2)$ from~\eqref{eq:lambdad1d2}. 
Since $\eps = \gamma^2/4 <\gamma$ we have
\[
\lambda = \frac{(1-\gamma)\e}{d}  \le \frac{(1-\eps)\e}{d} \le \frac{\e}{(1+\eps)d}\le \lambda\big( (1-\eps)d, (1+\eps)d \big) 
\]
where the last inequality follows from~\eqref{eq:lambdad1d2simple}.

Now that we know that $g_{(1-\eps)d} \circ g_{(1+\eps)d}$ has a unique fixed point,
it suffices to show that for $y=\frac{-1}{(1+\eps)d}$ 
\begin{equation}
\label{eq:wishful}
g_{(1-\eps)d} \left( g_{(1+\eps)d}\left(y\right) \right) \ge y \, .
\end{equation}
In that case, the fixed point $a^*$ is at least $\e^{y} \ge 1+y =
1-\frac{1}{(1+\eps)d}$. 
Since
\[
-x \le - \ln (1+x) \le -x \left( 1-\frac{x}{2} \right) \, , 
\]
whenever $x>0$, 
we have
\[
 g_{(1+\eps)d}\left(y\right) = -\ln(1+\lambda \e^{(1+\eps)dy}) = -\ln\left(1+{\lambda/\e}\right) \le -(\lambda/\e)(1-\lambda/2\e)
\]
and for any $z$,
\[
g_{(1-\eps)d}(z) = -\ln(1+\lambda\e^{(1-\eps)dz}) \ge -\lambda\e^{(1-\eps)dz}\,.
\]
Since $g_{(1-\eps)d}$ is monotonically decreasing, recalling $\lambda = \frac{(1-\gamma)\e}{d}$, we have
\begin{align*}
g_{(1-\eps)d} \left( g_{(1+\eps)d}\left(\frac{-1}{(1+\eps)d}\right) \right) &\ge
g_{(1-\eps)d} \left( -(\lambda/\e)(1-\lambda/2\e)  \right)\\
&\ge -\lambda \e^{-d(1-\eps) (\lambda/\e)(1-\lambda/2\e)} \\
&= \frac{(\gamma-1)}{d} \e^{1-(1-\eps)(1-\gamma)\left(1- \frac{1-\gamma}{2d}\right)} \, .
\end{align*}
Thus to prove~\eqref{eq:wishful} it suffices to show that 
\[
\frac{1}{1+\eps} \ge (1-\gamma) \e^{1-(1-\eps)(1-\gamma)\left(1- \frac{1-\gamma}{2d}\right)}
\]
or equivalently, setting $\eps = \gamma^2/4$,
\begin{equation}
\label{eq:wishful3}
- \ln(1+\gamma^2/4) \ge \ln(1-\gamma) +1 -(1-\gamma^2/4)(1-\gamma)\left(1- \frac{1-\gamma}{2d}\right) 
\end{equation}

We choose $d_0$ such that 
$ \frac{1-\gamma}{2d_0} < \gamma^3/3$. 
Then, recalling that $\ln(1-\gamma) = -\sum_i \gamma^i/i$,  for all $d\ge d_0$,  we have
\begin{align*}
1 + \ln(1-\gamma) + \ln(1+\gamma^2/4) &\le 1 -\sum_{i}\frac{\gamma^i}{i} +\frac{\gamma^2}{4}\\
&\le 1 -\gamma -\frac{\gamma^2}{4} -\frac{\gamma^3}{3}\\ 
&\le 1 -\gamma -\frac{\gamma^2}{4}  -\frac{1-\gamma}{2d} \\
&\le (1-\gamma^2/4)(1-\gamma)\left(1-\frac{1-\gamma}{2d}\right)
\end{align*}
which implies \eqref{eq:wishful3}.
\end{proof}

Let $a_{v, R}^{\mathbf{0}}$ and $a_{v, R}^{\mathbf{1}}$ denote the non-occupation probabilities at the root of the $\Poi(d)$ tree with activity $\lambda = \frac{(1-\gamma)\e}{d}$ when the vertices at depth $R$ are all occupied or all unoccupied respectively. 

We are now ready to prove
\begin{theorem}\label{thm:asymplb}
For all $\gamma \in (0, 1)$, for all sufficiently large $d$,
for all $\delta \in (0,1)$, 
there exists $R_0$ such that
\[
\Pr \left( (\forall R \ge R_0)
|a_{v, R}^{\mathbf{0}} -a_{v, R}^{\mathbf{1}}| \le
\e^{-\gamma^2R/56} \right) \ge 1 - \delta .
\]  
\end{theorem}

Fix $\gamma \in (0,1)$. Let $\lambda = \frac{(1-\gamma) \e}{d}$, and, as before, let $\eps = \gamma^2/4$. 
Denote $h = 1+ \left\lceil \frac{2\log \gamma -4}{\log(1-\gamma/2)} \right\rceil$.

We'll call a vertex $u$ in the $\Poi(d)$ tree 
\emph{good} if its subtree to depth $2h$ is $[(1-\eps)d, (1+\eps)d]$-regular. 
Note that for a Poisson random variable $X$ with mean $d$,  
and $0<\eps \le 1$, the following Chernoff bound holds:
\[
\Prob\left(|X-d|> \eps d\right) \le 2 \e^{-\eps^2 d/3}
\] 
(This follows, e.g., from \cite[Theorem 5.4 and inequalities (4.2), (4.5)]{mitz-upf}.)
Applying this to the vertex degrees in the subtree of depth $2h$
rooted at $u$, and taking a union bound, we find
\[
\Prob\left(\mbox{$u$ is good}\right) \ge 1 - 2 \left((1+\eps)d\right)^{2h+1} 
\e^{-\eps^2d/3}  
\ge 1 - \e^{-\eps^2 d/4}
\]
for all sufficiently large $d$.

\begin{lemma}\label{lem:good}
If $u$ is a good vertex then, subject to any boundary condition at
least $2h$ levels below $u$, we have  $1-a_u \le  \frac{1}{e^{\eps/3}d}$
\end{lemma}
\begin{proof}
Since the tree of depth $2h$ rooted at $u$ has even depth, $a_u$ is minimized 
when all its descendents at depth $2h$ below it are set to 0. Let $a_u^0$ be 
this minimum value, and let $a_u'$ be the non-occupation probability at $u$ of 
the $( (1-\eps)d, (1+\eps)d)$ alternating tree of height $2h$ rooted at $u$, 
when all its leaves are set to 0.

By pruning and grafting (Lemma~\ref{lem:prune-graft}), we know that $a_u^0 \ge a_u'$.

By Lemma~\ref{lem:fp-lowerbd}, the fixed point $a_*$ of  
$f_{(1-\eps)d}\circ f_{(1+\eps)d}$ is at least
$1- \frac{1}{(1+\eps)d}$.

Let $c$ be the constant from Lemma~\ref{lem:d1d2contraction} for 
$f_{(1-\eps)d}\circ f_{(1+\eps)d}$. Since $c >0$, by \eqref{eq:lamc}
\begin{align*}
\lambda &= c  ((1+\eps)d)^{(1-\eps)d} \left( \frac{(1-\eps)d+1}{(1-\eps)d(1+\eps)d  - c} \right)^{(1-\eps)d+1} \\
&\ge c  \frac{1}{(1+\eps)d} \left( \frac{(1-\eps)d+1}{(1-\eps)d} \right)^{(1-\eps)d+1}\\ & \ge \frac{c\;\e}{(1+\eps)d}
\end{align*}
whence it follows that   
\[
c \le \frac{(1+\eps)d\lambda}{\e} = (1-\gamma)(1+\eps) \le 1-\gamma/2  ,
\]
since by definition, $\eps = \gamma^2/4$.  By our choice of $h$, 
it follows that $c^{h-1} \e < \gamma^2/8 = \eps/2$.

Since $a_u' = f_{(1-\eps)d}\circ f_{(1+\eps)d}(0)$, by 
Lemma~\ref{lem:d1d2contraction} it follows that 
\begin{align*}
\left| a_u' - a_* \right| &\le c^{h-1} \ln (1+\lambda) \\
&\le c^{h-1} \lambda\\
&= \frac{c^{h-1}\e (1-\gamma)}{d}\\
&\le \frac{\eps/2}{(1+\eps)d} \,.
\end{align*}
Rearranging terms, we see that 
\[
a_u \ge a_u^0 \ge  a_u' \ge a_* - \frac{\eps/2}{(1+\eps)d} =
1- \frac{1+\frac12\eps}{(1+\eps)d} \,. 
\]
Finally,
\[
1-a_u \le  \frac{1+\frac12\eps}{(1+\eps)d} = 
\frac{1}{d}\left(1-\frac{\eps}{2(1+\eps)}\right) \le \frac{\e^{-\eps/3}}{d}
\]
whence the lemma follows.
\end{proof}

Consider any path $P$ from the root to a leaf at depth $R$ in the truncated 
$\Poi(d)$ tree. Fix $j \in \{0, 1, \dots, 2h-1\}$. Let $P_j = \{ 
u \in P | \depth(u) \equiv j \pmod{2h}\}$. For $u \in P_j$,  the events that 
$u$ is bad are independent. 

Let $X_{P,j}$ denote the number of bad $u$ in $P_j$. Then 
$\Exp X_{P,j} \le (R/(2h))\e^{-\eps^2 d/4}$, and by Chernoff's bound, 
for any $\alpha > 1$, 
\[
\Pr \left( X_{P,j} \ge \alpha \frac{R}{2h} \e^{-\eps^2 d/4} \right) \le \left(
  \frac{\e}{\alpha}\right)^{\alpha (R/(2h)) e^{-\eps^2 d/ 4}}
\]
Choosing $\alpha = \eps \e^{\eps^2 d / 4} / 4 \log (d)$, which is
exponential in $d$, we see that the right hand side becomes
\[
\left(\frac{\e}{\alpha}\right) ^ {\eps R/ 8 h \log(d)}
\]
In particular, for sufficiently large $d$,
this is less than
\[
\e ^ {-\eps^3 d R/ 40 h \log(d)} 
\]
This is so tiny that, even if we take a union bound over all $R$, 
all $j \le 2h$, and the ``first'' $\frac{4R^2}{\delta} d^R$ paths of length $R$ from
the root, the resulting probability bound still can be made smaller
than $\delta/2$.

Applying Markov's inequality to the expected number of nodes at depth
$R$,
we get that, with probability $\ge 1 - \delta/2$, there are at most
$\frac{4R^2}{\delta}d^R$ of these, for $R \ge R_0(\delta)$.  Thus, our union 
bound actually covered all the vertices at depth $R$.

Let $X_P= \sum_j X_{P,j}$ denote the total number of bad nodes on the path $P$.
Assuming the above ``good'' event, we have for all $R \ge R_0$, and
all paths $P$ of length $R$, that $X_P \le \alpha R \exp(-\eps^2d /4)$.

Let $w$ denote the leaf at depth $R$ on $P$ and $v$ denote the root. 
Recall that
\[
\left|\frac{\partial a_v}{\partial a_w}\right| \le (1+\lambda) 
\prod_{u \in P} (1- a_u)
\]
By Lemma~\ref{lem:good} we have 
\begin{align*}
\left|\frac{\partial a_v}{\partial a_w}\right| 
&\le 
\left(\frac{1}{\e^{\eps/3}d}\right)^{R-X_P}
\le \left(\frac{1}{\e^{\eps/3}d}\right)^{R (1 - \alpha \exp(-\eps^2
  d/4))}  \\
&= \left(\frac{1}{\e^{\eps/3}d}\right)^{R (1 - \eps/4 \log(d))}
\le d^{-R} \exp(-\eps R/3  + \eps R/4 + \eps^2 R / 12 \log(d))  \\
&\le d^{-R} \exp(-\eps R/13).
\end{align*}
By \eqref{eq:av0av1bd}, it follows that
\[
| a_v^{\mathbf{0}} - a_v^{\mathbf{1}} | \le | \bd T_R | d^{-R}
\exp(-\eps R/13) \le \exp(-\eps R/14),
\]
as desired, again assuming our good event, and noting that 
this implied $| \bd T_R | \le d^R \mathrm{poly}(R)$. This completes 
the proof of Theorem~\ref{thm:asymplb}.

Theorem~\ref{thm:asymplb} says that for any $\gamma \in (0,1)$, for 
sufficiently large $d$ the $\Poi(d)$ tree with activity $\frac{(1-\gamma)\e}{d}$
exhibits weak spatial mixing at the root, with probability 1. In other words, with probability 1, there is a well-defined value $a_v$, where $v$ is the root.
Moreover, since each node $w$ is the root of its own $\Poi(d)$ subtree, whose structure determines $a_w$, and there are only countably many nodes, it follows that, with probability 1, every node $w$ has a well-defined value $a_w$.

Since $a_w$ is the probability that $w$ is unoccupied, conditioned on its parent $p(w)$ being unoccupied, it follows that the occupation probabilities satisfy the recurrence
\[
\Pr( w \in X ) = (1-a_w) (1 - \Pr( p(w) \in X)),
\]
and hence, by induction on $\mathrm{depth}(w)$, these probabilities are well-defined, i.e. the $\Poi(d)$ tree exhibits weak spatial mixing at all vertices, with probability 1.  This completes the proof of Theorem~\ref{thm:asymplb-main}.

\section{Conclusion}

\paragraph*{Acknowledgments.}  
This work was supported in part by the National Science Foundation under Grant No. PHYS-1066293 and the hospitality of the Aspen Center for Physics. 
This work was also partially supported by NSF grants CCF-1150281 and CCF-1219117.

\end{document}